\documentclass[12pt,twoside]{article}

    \usepackage{graphicx}
    \usepackage{styleset}
    \usepackage{macros}
    \let\subsubsection\subparagraph

    \title  {Instantons and Bar-Natan homology}

    \author {P. B. Kronheimer and T. S. Mrowka%
      \thanks{%
        The work of the first author was supported by the National
        Science Foundation through NSF grants
        DMS-1405652 and DMS-1707924. The work of the second author was supported by
        NSF grants DMS-1406348 and DMS-1808794, and by a grant from the Simons Foundation,
        grant number 503559 TSM.}}

    \address {Harvard University, Cambridge MA 02138 \\
              Massachusetts Institute of Technology, Cambridge MA 02139}

\begin{document}

\maketitle

\begin{abstract}
A spectral sequence is established, whose $E_{2}$ page is Bar-Natan's
variant of Khovanov homology and which abuts to a deformation of
instanton homology for knots and links. This spectral sequence arises
as a specialization of a spectral sequence whose $E_{2}$ page is a
characteristic-2 version of $F_{5}$ homology, in Khovanov's
classification. 

\end{abstract}

\tableofcontents

\section{Introduction}

\subsection{Local coefficients}
In previous work, the authors introduced an instanton homology, $\Isharp(K)$,
for knots and links $K \subset S^{3}$. It was constructed as the Morse homology
of a Chern-Simons functional whose critical points correspond to certain
$\SU(2)$ representations of the fundamental group of the link complement. A variant
$\Jsharp(K)$ was introduced in \cite{KM-jsharp}, and was defined similarly, but
with $\SO(3)$ in place of $\SU(2)$. The coefficient ring in the
present paper will have characteristic $2$, and when this is
the case, both $\Isharp(K)$ and $\Jsharp(K)$ can be defined for webs
(embedded trivalent graphs) rather than only for links.  One of the
main results \cite{KM-unknot}
concerning $\Isharp(K)$ is the existence of a spectral sequence,
abutting to $\Isharp(K)$ and having $E_{2}$ page isomorphic to
Khovanov homology:
\begin{equation}\label{eq:kh-to-Isharp-ss}
          \kh(\bar K) \implies \Isharp(K)
\end{equation}
(The notation $\bar K$ denotes
the mirror image of $K$, and it appears here only because some of 
the traditional orientation conventions differ.)

In this paper $K$ will nearly always be a knot or link: trivalent
spatial graphs appear only in an auxiliary role.
We focus on a variant of $\Isharp(K)$ obtained by
introducing a system of \emph{local coefficients} on the relevant
configuration space of connections, $\bonf^{\sharp}(K)$. In doing so,
we build on two earlier papers. First, in \cite{KM-s-invariant}, the
authors introduced a local coefficient system, denoted here by
$\Gamma_{o}$.  It is defined as the pull-back of a local system on
$S^{1}$ via a map
\[
      h_{o} : \bonf^{\sharp}(K) \to S^{1}
\]
which in turn is defined using the holonomy of the connection along
$K$. In characteristic $0$, a spectral sequence similar to
\eqref{eq:kh-to-Isharp-ss} was established abutting to $\Isharp(K ;
\Gamma_{o})$, where the role of $\kh(\bar K)$ is
taken by Lee homology, a certain deformation of Khovanov homology
introduced in \cite{Lee-1, Lee-2}. The local system $\Gamma_{o}$ is a
system of free modules of rank $1$ over the ring $\Q[\Z] =
\Q[u,u^{-1}]$, though we will see later how it may be defined also in
characteristic $2$.

Second, in \cite{KM-deformation}, a local system $\Gamma_{\theta}$ was
introduced. Its construction is similar to $\Gamma_{o}$, but makes use of the
holonomy along the three edges of an auxiliary $\theta$-graph to
define a map
\[
         h_{\theta} : \bonf^{\sharp}(K) \to S^{1} \times S^{1} \times S^{1}.
\]
The result is a system of free rank-1 modules over the ring 
\begin{equation}\label{eq:Z3-ring}
  \F_{2}[\Z^{3}] = \F_{2}[T_{1}^{\pm 1}, T_{2}^{\pm 1}, T_{3}^{\pm 1}],
\end{equation}
where $\F_{2}$ is the field of 2 elements.

In this paper, we introduce a local system $\Gamma$ that generalizes both
$\Gamma_{o}$ (in its characteristic $2$ version) and $\Gamma_{\theta}$. It is a local system of rank-1
modules over a ring of Laurent series in $4$ variables:
\begin{equation}\label{eq:R-def}
            \cR = \F_{2} [T_{0}^{\pm 1}, T_{1}^{\pm 1}, T_{2}^{\pm 1}, T_{3}^{\pm 1}].
\end{equation}
The local system $\Gamma_{o}$ can be recovered as a specialization of $\Gamma$ by setting
$T_{i}=1$ for $i=1,2,3$, while the local system $\Gamma_{\theta}$ is
obtained by setting $T_{0}=1$.

\subsection{A spectral sequence from \texorpdfstring{$F_{5}$}{F5} homology}
Lee homology, mentioned above, is a member of a larger
family of a deformations of Khovanov homology
which are classified in \cite{Bar-Natan} and
\cite{Khovanov-Frobenius}. In the language and notation of
\cite{Khovanov-Frobenius}, these are link homologies $\mathsf{H}(K ;
F)$ arising from rank-2 Frobenius systems $F$. Among these, one that
is shown to be
universal in a particular sense arises from a Frobenius system
over
the ring $\Z[h,t]$. We will work exclusively in characteristic $2$ and
in place of Khovanov's $\Z[h,t]$ we introduce the ring
\[
       R_{5} = \F_{2}[h,t]
\]
and a corresponding Frobenius system $F_{5}$ whose underlying ring is
$R_{5}/(X^{2} + h X +t )$ and whose comultiplication is given by
\[
    \begin{aligned}
       \Delta: 1&\mapsto 1\otimes X + X\otimes 1 + h(1\otimes 1) \\
        \Delta: X&\mapsto X\otimes X + t(1\otimes 1).
    \end{aligned}
\]
(The subscript $5$ in $R_{5}$ and $F_{5}$ 
follows the naming convention in \cite{Khovanov-Frobenius}.)
The corresponding link homology is denoted here by $\mathsf{H}(K; F_{5})$. It
is a module over $R_{5}$ and is equal to $F_{5}$ when $K$ is the
unknot. The first topic of this paper is the construction of an instanton
homology $\Isharp(K ; \Gamma)$ corresponding to the local system of $\cR$-modules
$\Gamma$ described above, and the construction of the following
spectral sequence.

\begin{theorem}
   \label{thm:F5-ss} 
For a knot or link in\/ $\R^{3}$, there is a spectral sequence of\/ $\cR$-modules,
from the $F_{5}$ homology (in characteristic $2$) to the instanton
homology with local coefficients:
\begin{equation}\label{eq:F-to-Isharp-ss}
          \mathsf{H}(\bar K ;\, F_{5}  \otimes_{r} \cR ) \implies \Isharp(K; \Gamma).
\end{equation}
Here the base-change homomorphism $r: R_{5}\to \cR$ is given by
\[
\begin{aligned}
    r(h) &= P \\
    r(t) &= Q
\end{aligned}
\]
where
\begin{equation}\label{eq:PQ-formulae}
\begin{aligned}
        P &= T_{1}T_{2}T_{3} + T_{1}T_{2}^{-1}T_{3}^{-1} +
        T_{2}T_{3}^{-1}T_{1}^{-1} + T_{3}T_{1}^{-1}T_{2}^{-1}
\end{aligned}
\end{equation}
and
\[
    Q = \sum_{j=0}^{3} (T^{2}_{j} + T_{j}^{-2}).
\]
\end{theorem}

\begin{remark}
Because $\cR$ is a free module over $R_{5}$, one can take the tensor
product outside and rewrite the spectral sequence as
\[
          \mathsf{H}(\bar K ; F_{5}  )  \otimes_{r} \cR \implies \Isharp(K; \Gamma).
\]    
\end{remark}

\subsection{Bar-Natan homology}
By base-change of the coefficient ring via a  further ring homomorphism $\s:\cR\to\cS$, one obtains specializations
of the spectral sequence \eqref{eq:F-to-Isharp-ss}:
   \begin{equation}\label{eq:F-to-Isharp-ss-bc}
          \mathsf{H}(\bar K ;\, F_{5}  \otimes_{\s\comp r} \cS ) \implies \Isharp(K;\, \Gamma\otimes_{\s}\cS).
\end{equation}
As a particular case of this construction, we can obtain a spectral
sequence  from the \emph{graded Bar-Natan} link homology  $\BN(K)$
introduced in \cite{Bar-Natan}. There is in fact some freedom in the
construction of such a spectral sequence. To explain this, recall that in the context of
\cite{Khovanov-Frobenius} and \cite{Bar-Natan}, the homology $\BN(K)$
arises from the Frobenius system $F_{5}$ by a base-change
\[
            \tau_{\bn} : R_{5} \to S_{\bn}
\]
where $S_{\bn}=\F_{2}[h]$ and $\tau_{\bn}$ is the homomorphism
sending $t$ to $0$. 
We write $F_{\bn} = F_{5}\otimes_{\tau_{\bn}} S_{\bn}$ for the corresponding Frobenius
system, so that $\BN(K)$ is short-hand for $\mathsf{H}(K;
F_{\bn})$. Specifically, the underlying algebra of the Frobenius
system $F_{\bn}$ is
\[
           S_{\bn}[X] / (X^{2} + hX)
\]
and the comultiplication is
\[
    \begin{aligned}
       &\Delta:& 1&\mapsto 1\otimes X + X\otimes 1 + h (1\otimes 1) \\
        &\Delta:& X&\mapsto X\otimes X.
    \end{aligned}
\]

At the expense of working over a larger ring than
$S_{\bn}$, we can equivalently consider any ring homomorphism
\[
         \tau : R_{5} \to S
     \]
with the following two properties:
\begin{enumerate}
\item the polynomial $x^{2} + \tau(h) x + \tau(t)$
factorizes:
\begin{equation}\label{eq:bn-factors}
                    x^{2} + \tau(h) x +  \tau(t) = (x+a)(x+a'), \quad
                    (a,a' \in S);
                \end{equation}
\item \label{item:free-h} the ring $S$ is a free
    module over $S_{\bn}=\F_{2}[h]$ via the homomorphism $\tau_{1}:
    h\mapsto \tau(h)$.
\end{enumerate} 
When  factorization occurs, the
Frobenius system $F_{5}\otimes_{\tau} S$ can be described in terms of
a new generator $M = X+a'$, and the algebra becomes 
\[S[M] / (M^{2} + \tau(h) M).\] Thus the ``$t$'' term disappears from
the characteristic polynomial of $M$. The comultiplication is
  \[
    \begin{aligned}
       &\Delta:& 1&\mapsto 1\otimes M + M\otimes 1 + \tau(h)(1\otimes 1) \\
        &\Delta:& M&\mapsto M\otimes M.
    \end{aligned}
\]
When condition \ref{item:free-h} holds, an application of the
universal coefficient theorem shows that
\[
        \mathsf{H}(K ; F_{5}\otimes_{\tau} S) \cong \BN(K)\otimes_{\tau_{1}} S
\]
That is, the link homology arising from the Frobenius system
$F_{5}\otimes_{\tau} S$ is isomorphic to the graded Bar-Natan homology
with the coefficients extended trivially from $S_{\bn}=\F_{2}[h]$ to $S$.

With this in mind, we return to the instanton homology $\Isharp(K;
\Gamma)$ as a module over $\cR$. Suppose we
find a ring $\cS$, and a base change
\[
      \s: \cR \to \cS,
\]
such that the counterparts of the two conditions above hold:
\begin{enumerate}
\item  \label{item:factors-Isharp} the polynomial $x^{2} + \s(P) x + \s(Q)$
factorizes in $\cS[x]$:
\begin{equation}
                    x^{2} + \s(P) x +  \s(Q) = (x+A)(x+A'), \quad
                    (A,A' \in \cS);
                \end{equation}
\item \label{item:free-h-Isharp} the ring $\cS$ is a free
    module over $S_{\bn}=\F_{2}[h]$ via the homomorphism $r_{1}:
    h\mapsto \s(P)$.
\end{enumerate} 
If we examine the spectral sequence \eqref{eq:F-to-Isharp-ss-bc} under
these conditions, we see from the observations above that the link homology that appears on the
left (the $E_{2}$ page of the spectral sequence) is isomorphic to
graded Bar-Natan homology with a trivial extension of coefficients:
\[
       \mathsf{H}(\bar K ;\, F_{5}  \otimes_{\s\comp r} \cS ) \cong \BN(\bar
       K)\otimes_{r_{1}} \cS.
\]
In this way we obtain a spectral sequence from $\BN(\bar
K)\otimes_{r_{1}}\cS$ to $\Isharp(K; \Gamma \otimes_{\s} \cS)$.

To be specific about a base change that realizes the requirements
\ref{item:factors-Isharp} and \ref{item:free-h-Isharp}, we can
consider
\begin{equation}\label{eq:sbn-def}
        \s_{\bn} : \cR \to \cS_{\BN}
\end{equation}
where $\cS_{\BN}= \F_{2}[T_{1}^{\pm 1} , T_{2}^{\pm 1}, T_{3}^{\pm
  1}]$ is a ring of Laurent series in three variables, and
\begin{equation}\label{eq:sbn-base-change}
\begin{aligned}
    \s_{\bn}(T_{0})&=T_{1} \\
    \s_{\bn}(T_{i})&= T_{i} ,\qquad i=1,2,3.
  \end{aligned}
\end{equation}
We can write
\[
\begin{aligned}
    \s_{\bn}(P) &= A + A' \\
    \s_{\bn}(Q) &= AA',
\end{aligned}
\]          
where
\begin{equation} \label{eq:A-formulae}
    \begin{aligned}
        A &= T_{1}(T_{2}T_{3} + T_{2}^{-1}T_{3}^{-1})\\
        A' &= T_{1}^{-1}(T_{2}^{-1}T_{3} + T_{2}T_{3}^{-1})\\
    \end{aligned}
\end{equation}
so that the factorization \ref{item:factors-Isharp} indeed
occurs. Putting this together, we have the following statement.

\begin{corollary}\label{cor:BN-ss}
There is spectral sequence of modules over the Laurent series ring
$\cS_{\BN}$ in three variables,  from the graded Bar-Natan homology in
characteristic $2$,
\[
     \BN(\bar K) \otimes_{r_{1}} \cS_{\BN} \implies \Isharp(K ; \Gamma_{\BN}),
\]
to the instanton homology group with coefficients in the local system 
$\Gamma_{\BN}= \Gamma \otimes_{\s_{\bn}}\cS_{\BN}$, where the base
change $\s_{bn}$ is given by \eqref{eq:sbn-base-change}.
\end{corollary}

We shall also introduce a \emph{reduced} companion of the instanton
homology group $\Isharp(K ; \Gamma_{\BN})$, which we shall denote by
$\Inat(K; \Gamma_{\BN})$. The spectral sequence of
Corollary~\ref{cor:BN-ss} has a reduced companion, whose $E_{2}$ page
is the reduced Bar-Natan homology.  Such a reduced instanton homology group
can be defined using any local system of the form $\Gamma
\otimes_{\s} \cS$ provided that the base change $\s$ satisfies
$\s(T_{0})=\s(T_{1})$. In particular, there is no reduced version of
$\Isharp(K;\Gamma)$ itself. Correspondingly, there is no reduced
version of the link homology $\mathsf{H}(K; F_{5})$ without first making a
base change so that the polynomial $x^{2} + h x + t$ factorizes.

There are smaller rings $\cS$ that can be used in place of our $\cS_{\BN}$
in formulating this corollary. Notice that a sufficient condition for the factorization
\ref{item:factors-Isharp} to occur is that that $\s(Q)=0$.
So as another particular example we can take
\[
        \cS=\F_{2}[T,T^{-1}]
\]
and 
\[
\s(T_{i})= \begin{cases}1,&i=0,1\\ T, & i=2,3. \end{cases}
\]
Then $\s(P) = T^{2} + T^{-2}$ and $\s(Q)=0$.
The homomorphism $r_{1}:S_{\bn}\to \cS$ in this case is therefore given by
\[ \begin{aligned}
     r_{1}(h) &= T^{2} + T^{-2}.
\end{aligned}
\]

There is also a \emph{filtered} (as
opposed to graded) version of Bar-Natan homology which we denote by
$\FBN(K)$. It is obtained via further specialization from
$(R_{5}, F_{5})$
by setting $t=0$ and $h=1$. The result is a finite-dimensional
$\F_{2}$ vector space. For an instanton companion, we may pass to 
$\F_{4}$, the field of 4 elements
or any extension of $\F_{2}$ in which there is a solution $\zeta$ for the
equation $T^{2}+T^{-2}=1$. We define
\begin{equation}\label{eq:filtered-BN-s}
    \begin{aligned}
        &{\s}_{\fbn} : \cR \to\F_{4}\\
              & T_{i} \mapsto  \begin{cases}1,&i=0,1\\ \zeta, & i=2,3, \end{cases}
    \end{aligned}
\end{equation}
so that $Q\mapsto 0$ and  $P \mapsto 1$.
There is a corresponding local system of 1-dimensional $\F_{4}$-vector spaces,
\[
          \Gamma_{\FBN} = \Gamma\otimes_{\s_{\fbn}} \F_{4}.
 \]
We then have

\begin{corollary}\label{cor:barBN-ss}
For a knot or link $K$, let $\FBN(K)$ denote the filtered
Bar-Natan homology over $\F_{2}$. Then there is a spectral sequence of
vector spaces over\/ $\F_{4}$,
\[
  \FBN(\bar K)\otimes \F_{4} \implies \Isharp(K ; \Gamma_{\FBN}),
\]
where $\Gamma_{\FBN}$ is the local system of\/ $\F_{4}$ vector
spaces described above.    
\end{corollary}

Since $\Isharp(K;\Gamma)$  can be defined for trivalent spatial graphs
as well as knots and links, it would be interesting to know whether
there exist corresponding generalizations of the spectral sequence
\eqref{eq:F-to-Isharp-ss} or any of its specializations,
where the link homologies $\mathsf{H}(K; F_{5})$,  $\BN(K)$ or
$\FBN(K)$ are replaced by combinatorial invariants of
spatial graphs. Note however that $\Isharp(K;\Gamma)$ is a torsion
$\cR$-module when $K$ has vertices (it is annihilated by $P$), and
$\Isharp(K;\Gamma_{\FBN})$ is zero.

\section{The construction of  \texorpdfstring{$\Isharp(K ; \Gamma)$}{I(K;Gamma)}}
\label{sec:review}

In this section we describe the construction of the local system $\Gamma$ and the 
instanton homology $\Isharp(K ;
\Gamma)$. We will lean heavily on the expositions in the earlier papers
\cite{KM-deformation} and  \cite{KM-jsharp}, which were concerned with
two different specializations for $\Gamma$. 

\subsection{Instanton homology with constant coefficients}
A
trivalent graph, or \emph{web}, $K$ in a closed oriented $3$-manifold
$Y$ gives rise to an orbifold which we will simply denote by
$(Y,K)$. The isotropy groups are taken to be $\Z/2$ along edges of $K$
and $\Z/2 \times \Z/2$ at the vertices. We refer to such a special
orbifold as a \emph{bifold} and we consider orbifold $\SO(3)$ bundles
(or bifold bundles) $E$ over $(Y,K)$ requiring that the local isotropy
groups of the orbifold act effectively on the $\SO(3)$
fibers. \emph{Marking data} on $(Y,K)$ consists of an open set
$U_{\mu}$ and a bifold bundle $E_{\mu}\to U_{\mu} \setminus K$, and a
bifold connection is  \emph{marked} by $\mu$ if an isomorphism
$\sigma: E_{\mu} \to E|_{U_{\mu}}$ is given. An \emph{isomorphism} $\tau$ between
$\mu$-marked bundles with connection, $(E,A,\sigma)$,
$(E',A',\sigma')$ is an isomorphism of bifold bundles-with-connection such that the
automorphism $\sigma^{-1}\tau\sigma' : E_{\mu}\to E_{\mu}$ lifts to
the determinant-1 gauge group. We write
\[
   \bonf(Y, K ; \mu)
\]
for the space of isomorphism classes of $\mu$-marked bifold bundles
with connection.

 The marking data $\mu$ is \emph{strong} if the
automorphism group of every flat $\mu$-marked bifold connection is
trivial. A sufficient condition is that $U_{\mu}$ contains a vertex of
$K$, and in this case there are indeed no connections with non-trivial
stabilizer even in $\bonf(Y, K; \mu)$. 
With coefficients in the field $\F_{2}$ of two elements, one can
construct an $\SO(3)$ instanton homology group
\[
             J((Y,K) ; \mu) 
\]
for any bifold with strong marking data. The generators of the complex
from which this instanton homology is computed correspond to critical
points of a perturbed Chern-Simons functional on $\bonf(Y, K ; \mu)$. 
We may omit $Y$ from our notation for both $J$ and $\bonf$ when $Y$ is
understood
(which is often the case when $Y$ is $S^{3}$).

Consider next a framed base-point $y_{0} \in Y$ with standard
neighborhood $B(y_{0}) \cong B^{3}$. We write $Y^{o}$ for the
complement of this standard neighborhood:
\[
            Y^{o} = Y \setminus B(y_{0}).
\]
Given a web \[ K \subset
Y^{o},\] we may form a new web as  a split union
\begin{equation}\label{eq:union-with-theta}
       K^{\sharp} =  K \cup \theta
\end{equation}
where $\theta \subset B(y_{0})$ is a standard theta-graph (three edges and two
vertices) contained in the ball. We may then
define
\begin{equation}\label{eq:Jsharp-recap}
             \Jsharp(K) = J((Y, K^{\sharp}); \mu_{\theta})
\end{equation}
where the marking data $\mu_{\theta}$ consists of the ball $U_{\theta} = B(y_{0})$ containing
$\theta$, with $E_{\theta}$ the unique trivial
bundle on $B^{3}\setminus \theta$. The group $\Jsharp(K)$ was defined
first in \cite{KM-jsharp}, though the description in that paper was a slight
variant of this one. The description we have just given is from
\cite{KM-deformation}, where the equivalence of the two descriptions
is also proved.

In this paper, we will be almost exclusively concerned with the case
that the marking region $U_{\mu}$ is not just the ball $B(y_{0})$ but is instead
the whole of $Y$. The distinguished $\SO(3)$ bundle $E_{\mu}$ on
$Y\setminus K^{\sharp}$ may in general have non-zero Stiefel-Whitney
class
\[
\begin{aligned}[]
    w_{2}(E_{\mu}) &\in H^{2} ( Y\setminus K^{\sharp} ; \Z/2).
\end{aligned}
\]
We take this class to be represented by $\omega \subset Y$, which is a
codimension-2 submanifold with boundary. 
We make the following assumptions on $\omega$:
\begin{itemize}
\item $\omega$ is (the interior of) a union of circles and arcs with end-points on
    $K$;
\item $\omega$ is disjoint from the ball $B(y_{0})$ which contains $\theta$.
\end{itemize}
We require that $\omega$ represent $w_{2}(E_{\mu})$, in the sense that
\[
      w_{2}(E_{\mu}) = \mathop{PD}[\omega \cap (Y\setminus K^{\sharp})].
\]
Having chosen $\omega$, we shall trivialize $E_{\mu}$ on the
complement of $\omega$, so that the obstruction to extending the
trivialization across each component of $\omega$ is non-zero. We use
this trivialization to give a lift of $E_{\mu}$ to an $\SU(2)$ bundle
on the complement of $\omega$.

Let us write $\mu_{\omega}$ for the marking data obtained in this way
from a 1-manifold $\omega \subset Y \setminus K^{\sharp}$.

\begin{definition}
Using the marking data $E_{\mu}$ as above, whose Stiefel-Whitney class
is dual to $\omega\subset Y\setminus K$, we write
    \begin{equation}\label{eq:Isharp-recap}
        \Isharp(K)_{\omega} = J((Y, K^{\sharp}); \mu_{\omega})
    \end{equation}
for the corresponding instanton homology of the web $K\subset
Y^{o}$.
    We also write
    \begin{equation}\label{eq:bonf-sharp}
        \bonf^{\sharp}(K)_{\omega} = \bonf((Y, K^{\sharp}) ; \mu_{\omega})
    \end{equation}
    for the corresponding configuration space of connections. 
When $\omega$ is empty, we simply omit it
from our notation, and write $\Isharp(K)$.
\end{definition}

When $K$ is a knot or link, this variant
coincides with $\Isharp(K)_{\omega}$ as introduced in \cite{KM-unknot} (though
in that paper the coefficient ring was $\Z$). As with $\Jsharp$, the
definition we have presented here is slightly different from the
earlier one: the difference is the use of the graph $\theta$ in place
of the Hopf link that was used in $\cite{KM-unknot}$. But the two
definitions give isomorphic homology groups, by the arguments from
\cite{KM-deformation}.

Because the marking data is all of $Y$, the gauge theory which
underlies this instanton homology is essentially an $\SU(2)$ gauge
theory.  In particular, we have the following identification:

\begin{lemma}\label{lemma:bonf}
    The space $\bonf^{\sharp}(K)_{\omega}$ parametrizes
    equivalence classes of data of the following sort:
    \begin{itemize}
    \item an $\SU(2)$ bundle $\hat E$ over
        $Y\setminus (K^{\sharp} \cup \omega)$;
    \item an $\SU(2)$ connection $\hat A$ in $\hat E$; subject to the
        restrictions,
    \item the associated $\SO(3)$ connection $A$ in the adjoint bundle
        of $\hat E$ is the restriction to
        $Y\setminus (K^{\sharp} \cup \omega)$ of a bifold bundle on
        the bifold $(Y, K^{\sharp})$;
    \item the limiting holonomy of $\hat A$ on small circles linking
        $\omega$ is $-1$.
    \end{itemize}
\end{lemma}

\subsection{The local system}
\label{subsec:local-system}

We begin with some motivation of our construction. 
If $\pi$ denotes the fundamental group of the configuration space
$\bonf^{\sharp}(K)_{\omega}$, then there is a tautological local system over
$\bonf^{\sharp}(K)_{\omega}$ whose fiber at each point is a free rank-1 module
for the group ring $\F_{2}[\pi]$. It can be realized by defining its
fiber at $[A] \in \bonf^{\sharp}(K)_{\omega}$ to be the vector space of
$\F_{2}$-valued functions with finite support on the fiber of the
universal cover $\widetilde{\bonf}^{\sharp}(K)_{\omega} \to \bonf^{\sharp}(K)_{\omega}$. Given any choice of homomorphism
\[
    \epsilon: \pi \to G
\]
there is a corresponding local system $\Gamma_{\epsilon}$ of $\F_{2}[G]$ modules.
 Our instanton homology is $\Z/2$
graded, not $\Z$ graded, because of non-trivial spectral flow, and
there is an infinite cyclic cover of $\bonf^{\sharp}(K)$ on which the
spectral flow is trivial. Let
\[
        \pi' \subset \pi
\]
be the fundamental group of this infinite cyclic cover.
The instanton homology groups are
essentially unchanged in passing to the cover (the homology becomes
$\Z$ graded and 2-periodic, rather than $\Z/2$ graded). Up to
isomorphism, the instanton homology groups with coefficients in the
local system $\Gamma_{\epsilon}$ will therefore depend on
the homomorphism $\epsilon$ only through the restriction of $\epsilon$ to
$\pi'$. 

Although we shall not need a proof, the fundamental group  $\pi$ is a
free abelian group of rank $5$ when $K$ is a knot and $\omega$ is
empty. It follows in this case that
$\pi'$ has rank $4$, and we will therefore capture the most general
local system
if we construct a homomorphism
\[
              \pi \to \Z^{4}
\]
which is injective on the subgroup $\pi' \subset \pi$.

 The construction described in
\cite{KM-deformation} arises from a map
\[
           \pi \to \Z^{3}
\]
presented in terms of an explicit map
\[
         (h_{1}, h_{2}, h_{3}): \bonf^{\sharp}(K)_{\omega} \to S^{1} \times S^{1} \times S^{1},
\]
and this leads to the local system $\Gamma_{\theta}$ of free rank-1
modules over $\F_{2}[\Z^{3}]$ as described at \eqref{eq:Z3-ring} in
the introduction.  To recall this briefly from \cite{KM-deformation}, 
the marking data $\mu_{\omega}$ means that our gauge
theory has structure group $\SU(2)$, and at the two vertices of
$\theta$, the structure group of the two fibers $E_{+}$ and $E_{-}$ 
is reduced to the center $\{\pm 1\}$. 
Along each edge of $\theta$, the structure group is reduced to $S^{1}$. The
holonomy along each edge therefore gives a well defined element of
\begin{equation}\label{eq:pm-thing}
              S^{1} / \{\pm 1\} \cong \R/\Z.
\end{equation}
Applied to three edges of $\theta$ in turn, one obtains the three components $h_{1}$,
$h_{2}$, $h_{3}$. (The notation for $\Gamma_{\theta}$ was
simply $\Gamma$ in
\cite{KM-deformation}).

When $\omega$ has no end-points (so is disjoint from $K$), 
a very similar construction from \cite{KM-s-invariant}, can be adapted to
define a map
\begin{equation}\label{eq:h0-map}
          h_{0} : \bonf^{\sharp}(K)_{\omega} \to S^{1}.
\end{equation}
To describe this, consider first the case that $K$ is a knot. Choose a
framing $\tau$ for $K$ so as to have well-defined push-off. As
explained in \cite{KM-singular}, the framing allows us to interpret the 
orbifold connection $[A]$ as giving rise to a well-defined connection
over the knot $K$ itself, carried by a bundle with an involution $g$
coming from the orbifold structure. Because of the action of $g$, the adjoint bundle 
decomposes as a sum \[ \xi \oplus \eta\] where $\xi$ is a  real line
bundle on $K$, and $\eta$ is a 2-plane bundle. The marking data allows
us to identify the orientation bundle of the knot $K$
with the orientation bundles of both $\xi$ and $\eta$, so the
connection in the 2-plane bundle $\eta$
has a well-defined circle-valued holonomy along $K$. The holonomy of 
$\eta$ around $K$ is the definition of $h_{0}$
above. If $K$ is a link rather than a knot, we multiply the holonomies
along all the components.

Combining the two previous constructions, we now have a map
\begin{equation}\label{eq:four-h-maps}
           (h_{0}, h_{1}, h_{2}, h_{3}) : \bonf^{\sharp}(K)_{\omega} \to \R^{4}/\Z^{4}
\end{equation}
whenever $\omega$ has no boundary points. In the case that $\omega$
has boundary on $K$, the component $h_{0}$ must be omitted.
As in \cite{KM-singular}, we use an explicit description of the
corresponding local system that depends on the maps $h_{i}$ but does
not depend on a choice of base-point in $\bonf^{\sharp}(K)_{\omega}$. We write
\[
                \cR = \F_{2}[ \Z^{4}]
\]
and regard this as a subring of $\F_{2}[\R^{4}]$. For each
$\mu\in\R^{4}$, we have the rank-1 $\cR$-module
\[
\begin{aligned}
    \Gamma|_{\mu} &= T_{0}^{\mu_{0}} T_{1}^{\mu_{1}} T_{2}^{\mu_{2}}
    T_{3}^{\mu_{3}} \cR \\
                 &\subset \F[\R^{4}]
\end{aligned}
\]
and these form a local system over the torus $\R^{4}/\Z^{4}$. Pulling
this back by the map \eqref{eq:four-h-maps}, we obtain our local
system $\Gamma$ over $\bonf^{\sharp}(K)_{\omega}$. 

We summarize these constructions as follows:

\begin{notation} \label{notation:Gamma}
    Let $K \subset Y^{o}$ be a link, let $K^{\sharp} = K \cup \theta$,
    let $\omega \subset Y$ be a 1-manifold without boundary, disjoint
    from the ball containing $\theta$. Let $\bonf^{\sharp}(K)_{\omega}$ be
    the associated space of connections. On $\bonf^{\sharp}(K)_{\omega}$,
    we have a local system $\Gamma$ of (free, rank-1) modules
    over \[\cR=\F_{2}[\Z^{4}].\]
    For any base-change $\s:\cR\to \cS$, there is a corresponding
    system of $\cS$-modules, \[ \Gamma_{\s} = \Gamma\otimes_{\s}\cS.\]
    If $\s(T_{0})=1$, then the map $h_{0} : \bonf^{\sharp}(K)_{\omega}
    \to S^{1}$
    is not required in the construction of the local system, and in
    this case we can form the local system $\Gamma_{\s}$ more
    generally, when $\omega$ is allowed to be a manifold with boundary
    with end-points on $K$.
\end{notation}

\begin{remarks}
    (i) In the case that $\omega$ has boundary and $\s(T_{0})=1$, our
    notation for the local system $\Gamma_{\s}$ involves a slight abuse
    of notation, since we can no longer write it as
    $\Gamma \otimes_{\s} \cS$. (The local system $\Gamma$ is not
    defined in this case.)  It should more properly be defined as
    $\Gamma_{\theta}\otimes_{\bar{\s}} \cS$, where $\bar{\s} :
    \F_{2}[\Z^{3}] \to \cS$ is the map through which $\s$ factors.

    (ii) The definition of $h_{0}$ above makes use of a framing $\tau$ for
    the knot (or for each component of the link). If the framing
    $\tau$ is changed by $1$, then $h_{0}$ changes to $h_{0} +
    1/2$. (See \cite{KM-singular}.) Therefore,  framings whose difference is even give rise to the
    same map $h_{0}$ and identical local systems $\Gamma$.

   (iii) The four maps $h_{i}$ in \eqref{eq:four-h-maps} give a map
    \[
               \phi: \pi = \pi_{1}(\bonf^{\sharp}(K)_{\omega}) \to \Z^{4}
    \]
    whenever $\omega$ has no boundary points.
    We can say a little more about the kernel and image of
   $\phi$. The space $\bonf^{\sharp}(K)_{\omega}$ is connected and can be
   identified as usual with a quotient $\cA / \cG$, of an affine space
   of connections by the action of the gauge group. We can therefore
   identify $\pi$ as $\pi_{0}(\cG)$. When $K$ is a knot, the kernel of the map $\phi$ is
   $\Z$ and 
   consists of the components of $\cG$ represented by gauge
   transformations that are supported in the neighborhood of a point
   in $S^{3}\setminus K$. The image of $\phi$ is a sublattice
   $\Lambda \subset \Z^{4}$ of index $8$. In terms of the standard basis
   $v_{i}$, this lattice is generated by the elements
   \[
               2v_{i},\; \text{($i=0,1,2,3$),} \quad \text {and} \quad
               v_{1} + v_{2} + v_{3}.
   \]
    For example, the fact that the $v_{0}$ coefficient is even is a
    reflection of the fact that the map $h_{0}$ lifts to the
    double cover of $S^{1}$. In turn this lift exists because we can
    use the holonomy of the $\SU(2)$ connection around the loop $K$ to define a map
    $\tilde h_{0}$, rather than use the holonomy of the $\SO(3)$
    connection which defines $h_{0}$. In the same way, the fact that the coefficients of
    $v_{1}$ and $v_{2}$ have the same parity similarly means that
    $h_{1}+h_{2}$ lifts to a double cover, essentially because the
    corresponding pair of edges of $\theta$ form a closed
    loop. Instead of the ring $\F_{2}[\Z^{4}]$, we could instead work
    with the subring $\F_{2}[\Lambda]$, which we can identify as the
    subring generated by the monomials $T_{i}^{\pm 2}$ and
    $T_{1}T_{2}T_{3}$.

     (iv) The previous remark explains that the circle-valued map
     $h_{0}$ has a lift $\tilde h_{0}$, through the double-cover of
     $S^{1}$, and this accounts for a difference in conventions
     between the present paper and (for example) the notation in
     \cite{KM-s-invariant}. In \cite{KM-s-invariant}, the local system
     is defined using the map $\tilde h_{0}$, and is described as a
     module for the ring of finite Laurent series in a formal variable
     $u$. Because of the double cover, the variable $u$ in that paper
     corresponds to $T_{0}^{2}$ in the present paper, rather than
     $T_{0}$. Our present choice of conventions is for consistency
     with \cite{KM-deformation}.

     (v) In the case that $K$ is an $n$-component link and $\omega$
     has no boundary points, the
     fundamental group $\pi$ is free of rank $4+n$ and $\pi'$ is free
     of rank $3+n$. If $\omega$ has boundary, and if $k$ is the number of components of $K$ on which
     $\partial \omega$ has an odd number of points, then the rank of
     $\pi'$ is $3+n-k$ and the torsion subgroup of $\pi'$ is
     $(\Z/2)^{k-1}$. The proof of these assertions are essentially the
     same as the results of section~3.2 of \cite{KM-unknot}.
\end{remarks}

\subsection{The chain complex}

Following Notation~\ref{notation:Gamma} henceforth, we fix a
base-change homomorphism  $\s:\cR\to\cS$, possibly the identity. We fix
a 3-manifold $Y$, a base-point $y_{o}\in Y$,  a link $K \subset Y$
disjoint from a fixed ball around the base-point, and a representative $\omega$
for the Stiefel-Whitney class. If $\s(T_{0})=1$, then we allow $\omega$
to have boundary on $K$.

We can now construct the chain complex and boundary map which will
define a Floer homology group $\Isharp(K;\Gamma_{\s})_{\omega}$ in the
usual way for an instanton Floer homology. While the construction is a
straightforward generalization of the treatment in
\cite{KM-deformation} and its predecessors, it is worthwhile to recall
a particular point from \cite{KM-deformation}, namely the proof that
$\partial^{2}=0$ given in \cite[Lemma 3.1]{KM-deformation}. From there
we see that, \emph{a priori}, there is a relation of the form
\[
         \partial \circ \partial = W \mathbb{1}
\]
for some $W\in \cS$. That is, we may have a ``matrix factorization''
rather than a complex. The proof that $W=0$ carries over from
\cite{KM-deformation} without change: it vanishes because it is a sum
of contributions from the vertices of $\theta$, and is independent of
$K$. Although we will not pursue this further in the present paper, it
is worth observing what happens in a more general situation. Suppose
we consider the case that $K$ is a web rather than a link, and suppose
we use each edge $e$ of $K$ to define a map $h_{e} :
\bonf^{\sharp}(K)\to S^{1}$, so as to obtain a local system of modules
over a ring of Laurent series in variables $T_{e}$ indexed by the
edges. The term $W$ will have contributions from possible bubbling at
the vertices $v$ of $K$, so
\[
         W = \sum_{v} W_{v}.
\]
With a little more care, one may explicitly compute $W_{v}$, and it
has the form
\[
            W_{v} = p( T_{e(v,1)},  T_{e(v,2)},  T_{e(v,3)})
\]
where $e(v,i)$ are the three edges incident at $v$ and 
\[
         p(T_{1}, T_{2}, T_{3}) = T_{1}T_{2}T_{3} + T_{1}T_{2}^{-1}T_{3}^{-1} +
        T_{2}T_{3}^{-1}T_{1}^{-1} + T_{3}T_{1}^{-1}T_{2}^{-1}
\]
is the same polynomial that defines $P$. (Our notation here as
elsewhere will sometimes not distinguish a generator $T_{i}\in \cR$
from its image under the base-change, $\s(T_{i})\in \cS$.) In this generality, the
potential $W$ is non-zero. It becomes zero if we impose relations on
the variables $T_{e}$ so as to ensure (for example) that the product
\[
 T_{e(v,1)} T_{e(v,2)} T_{e(v,3)}
\]
is independent of the vertex $v$.

\subsection{Functoriality}
\label{subsec:functoriality}

Let $X$ be an oriented 4-dimensional cobordism from $Y_{0}$ to $Y_{1}$
and let $S\subset X$ be a surface (not necessarily orientable) 
which provides a cobordism between links $K_{0}\subset Y_{0}$ and
$K_{1}\subset Y_{1}$. 
Because of the need for a
basepoint, we suppose that $X$ contains an embedded cylinder
$[0,1]\times B^{3}$ whose boundary at the two ends are the balls
$B(y_{0})$ and $B(y_{1})$ around the chosen basepoints. We suppose
that $S$ is disjoint from this cylinder, so we may form the larger
foam
\[
\begin{aligned}
    S^{\sharp} &= S \; \cup \; \bigl( [0,1]\times \theta \bigr)\\ & \subset
    X.
\end{aligned}
\]
The foam $S^{\sharp} \subset X$ provides an orbifold structure on $X$,
which we write as $(X,S^{\sharp})$. 

The oriented orbifold $(X,S^{\sharp})$ is a cobordism between the
orbifolds $(Y_{i}, K_{i}^{\sharp})$.  As a special case of the general
machinery of \cite{KM-jsharp} and \cite{KM-unknot}, it defines homomorphisms on the
constant-coefficient instanton homology groups
\[
          \Isharp(X,S) : \Isharp( Y_{0}, K_{0}) \to \Isharp(Y_{1}, K_{1}).
\]

More generally, we can again allow bundles with non-zero $w_{2}$ represented
as the dual of a submanifold $\omega$. As in \cite{KM-unknot}, we take
$\omega$ to be a surface with corners. Thus, the boundary of $\omega$
consists of:
\begin{itemize}
\item a 1-manifold $\omega_{0}\subset Y_{0}$, possibly with boundary
    on $K_{0}$;
\item a 1-manifold $\omega_{1}\subset Y_{1}$ similarly;  and
\item a union of arcs and circles in the surface $S\subset Y$.
\end{itemize}
As well as meeting $S$ along its boundary, we allow $\omega$ to meet
$S$ also in its interior, in transverse points of intersection
\cite{KM-unknot}. We then have the more general functoriality, with
maps
\[
          \Isharp(X,S)_{\omega} : \Isharp( Y_{0}, K_{0})_{\omega_{0}} 
\to \Isharp(Y_{1}, K_{1})_{\omega_{1}}.
\]

We can now introduce the local system $\Gamma_{\s}=
\Gamma\otimes_{\s}\cS$. 
If $\partial \omega$ meets $S$, then we require that $\s(T_{0})=1$, as
in Notation~\ref{notation:Gamma}.
When the local system  is introduced and $\s(T_{0})\ne 1$ we must also
take additional care,
because of the role of the framings. Recall that the map $h_{0} :
\bonf^{\sharp}(K) \to S^{1}$ depends on a choice of framing of $K$,
and otherwise has an ambiguity of a half-period. Framings which have different
parities give rise to groups $\Isharp(K ; \Gamma_{\s})_{\omega}$ that are
isomorphic, but not canonically so without further choices. This issue
is dealt with carefully in \cite{KM-singular} (and with some inessential
inaccuracies in \cite{KM-s-invariant}). We recall the procedure.

Let us recall first that the construction of $\Isharp(K;\Gamma_{\s})_{\omega}$
depends on framing of $K$, and that the map that we are seeking to
define should therefore be written
\begin{equation}\label{eq:Isharp-functor}
     \Isharp(X,S ; \Gamma_{\s})_{\omega} : \Isharp( Y_{0}, K_{0}; \Gamma_{\s})^{\tau_{0}}_{\omega_{0}} \to \Isharp(Y_{1}, K_{1}; \Gamma_{\s})^{\tau_{1}}_{\omega_{1}}
\end{equation}
where we have now included the framings of $K_{0}$ and $K_{1}$
explicitly in the notation.
At the chain level, the map will be given by a chain map
\[
       C^{\sharp}(S; \Gamma_{\s})_{\omega} :  C^{\sharp}( Y_{0}, K_{0}; \Gamma_{\s})^{\tau_{0}}_{\omega_{0}} \to C^{\sharp}(Y_{1}, K_{1}; \Gamma_{\s})^{\tau_{1}}_{\omega_{1}}
\]
whose matrix entry from $\alpha_{0}$ to $\alpha_{1}$ is
given by ``counting instantons'' as usual,  and attaching a ``weight''
$\epsilon([A]) \in \cR$ to each instanton $[A]$ from $\alpha_{0}$ to $\alpha_{1}$.

To define $\epsilon([A])$, following \cite{KM-singular}, we note that the framings of $K_{0}$ and $K_{1}$ allow one to define a
self-intersection number $S\cdot S$ for the surface $S$. 
Given the instanton $[A]$ on the cylindrical-end
orbifold obtained from $(X,S)$, one may define locally an $\SO(3)$
bundle on $S$ with a reducible connection, so that the associated
$\R^{3}$ bundle has locally the form $\xi\oplus \eta$ (a line bundle
and a 2-plane bundle). Although the construction
is local, the curvature $2$-form of the 2-plane bundle $\eta$ exists
globally on $S$ as a 2-form $\Omega$ with values in the orientation
bundle of $S$. So we may consider the integral
\[
         \nu_{0}(A) = c \int_{S} \Omega,
\] 
with the normalizing constant $c$ chosen so that $\nu_{0}$ coincides with
the Euler class of $\eta$ in the closed orientable case. (So $c= i/2\pi$
if we identify the Lie algebra of the circle with $i\R$ in the usual
way.)  An application of
Stokes theorem gives
 \begin{equation}\label{eq:nu-stokes}
               \nu_{0}(A) + (1/2)(S\cdot S) = h_{0}(\alpha_{1}) -  
             h_{0}(\alpha_{0})\pmod \Z.
\end{equation}
We may define similar quantities $\nu_{i}(A)$, ($i=1,2,3$), as the integrals
of the curvature on the three edges of $\theta$.
The weight $\epsilon([A])$ is now defined by
\begin{equation}\label{eq:T-cobordism}
        \epsilon([A]) = T_{0}^{\nu_{0}(A) + (1/2)(S\cdot S)} T_{1}^{\nu_{1}(A)}
 T_{2}^{\nu_{2}(A)}  T_{3}^{\nu_{3}(A)}
\end{equation}
These relation \eqref{eq:nu-stokes}, and the simpler formulae for the
other $\nu_{i}$, mean that multiplication by \eqref{eq:T-cobordism} is
a map from the fiber $\Gamma_{s,\alpha_{0}}$ to $\Gamma_{s,\alpha_{1}}$ as
required. The formula for the exponent of $T_{0}$ is the same as in
\cite{KM-singular}, except for a factor of $2$ which stems from the
difference between $\SU(2)$ and $\SO(3)$, as explained in Remark (ii)
at the end of section~\ref{subsec:local-system} above.

The result of this construction is a well-defined map \eqref{eq:Isharp-functor}
between instanton homology groups. As a special
case, we may take $K_{0}=K_{1}=K$ and use the cylindrical cobordism to
obtain canonical isomorphisms
\[
 \Isharp( Y, K; \Gamma_{\s})^{\tau_{0}}_{\omega} \to \Isharp(Y, K; \Gamma_{\s})^{\tau_{1}}_{\omega}
\]
where only the framing has changed. We use these canonical
isomorphisms to treat $\Isharp(Y, K; \Gamma_{\s})^{\tau}_{\omega}$ as being
independent of the choice of framing $\tau$. Note however, that if
$\tau_{0}$ and $\tau_{1}$ are framings which are equal mod $2$, then
the corresponding local systems are identical, but our chosen
canonical isomorphism is \emph{not} the identity map: it is
multiplication by $T_{0}^{n}$, where $n$ is half the difference
between the framings.

Henceforth we will continue to omit $\tau$ from our notation. When the
ambient cobordism $X$ is a cylinder, or is otherwise understood, we
will simply write
\begin{equation}\label{eq:functorial}
     \Isharp(S ; \Gamma_{\s})_{\omega} : \Isharp(  K_{0}; \Gamma_{\s})_{\omega_{0}} \to \Isharp( K_{1}; \Gamma_{\s})_{\omega}
\end{equation}
for the map \eqref{eq:Isharp-functor}.

\subsection{A reduced variant}\label{subsec:reduced}

The instanton homology $\Isharp(K)$ with constant coefficients has a
``reduced'' version $\Inat(K)$, which is described, for example, in
\cite{KM-unknot}. It depends on a choice of base-point on the link
$K$, and the relationship of $\Inat$ to $\Isharp$ is similar to the
relationship between the reduced and unreduced versions of Khovanov
homology \cite{Khovanov}. Given a base change $\s: \cR\to\cS$
satisfying the extra condition $\s(T_{0})=\s(T_{1})$, we can
construct a reduced version $\Inat(Y,K;\Gamma_{\s})_{\omega}$ of
$\Isharp(Y,K;\Gamma_{\s})_{\omega}$, for knots and links  $K\subset
Y$. We describe the construction here.

\begin{figure}
    \begin{center}
        \includegraphics[scale=0.3]{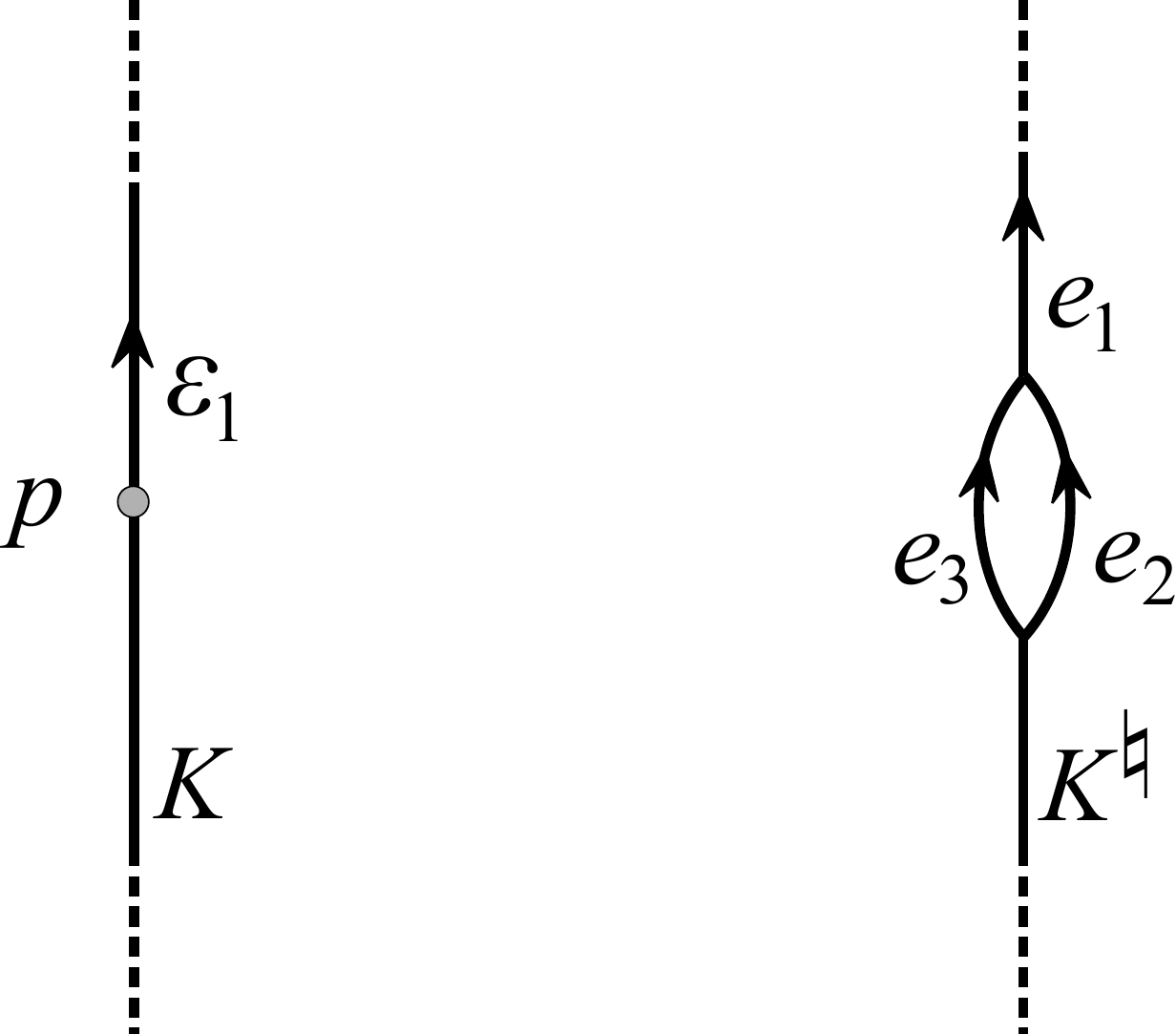}
    \end{center}
    \caption{\label{fig:reduced}
    A knot $K$ with base-point $p$, and the resulting web
    $K^{\natural}$ obtained by
    adding a bigon.}
\end{figure}

Let $K\subset Y$ be a link. Let $p$ be the base-point on $K$, and $(\epsilon_{1}, \epsilon_{2},
\epsilon_{3})$ be an oriented basis of tangent vectors, with
$\epsilon_{1}$ pointing along $K$. Making the modification in a
standard ball around $p$, create a spatial graph with two
vertices by replacing an arc of $K$ adjacent to $p$ with a bigon, as
shown in Figure~\ref{fig:reduced}. Let $K^{\natural}\subset Y$ denote the
resulting web. Let
$\bonf^{\natural}(Y, K)$, or just $\bonf^{\natural}(K)$, denote the
space of marked bifold connections on the corresponding bifold.

We define three circle-valued functions,
\[
      (h_{1}, h_{2}, h_{3}) : \bonf^{\natural}(K) \to \R^{3}/\Z^{3},
\]
as follows. First, in the case that $K$ is knot, the web
$K^{\natural}$ is the union of three oriented arcs $e_{1}$, $e_{2}$,
$e_{3}$, where $e_{2}$ and $e_{3}$ comprise the added bigon. All three
are oriented by $\epsilon_{1}$. As before structure group of a
connection $[A]\in \bonf^{\natural}(K)$ reduces to $S^{1}$ along the
arcs, and the holonomy of $[A]$ along the three arcs defines the maps
$h_{i}$, just as in the case of $\bonf^{\sharp}(K)$. If $K$ is a link,
let 
\[
         h_{0} : \bonf^{\natural}(K) \to \R/\Z
\]
be obtained from the holonomy along the remaining components of $K$
(those that do not contain $p$). In the case of a knot, just take
$h_{0}$ to be constant. In either case, we now have a map
\[
           (h_{0},h_{1}, h_{2}, h_{3}) \to \R^{4}/\Z^{4},
\]
from which we can construct a local system of $\cR$-modules $\Gamma$ over
$\bonf^{\natural}(K)$ as before. In the case that the base-change $\s:
\cR\to\cS$ has $\s(T_{0}) = \s(T_{1})$, the local system is pulled
back from $\R^{3}/ \Z^{3}$ via the map
\[
        \R^{4}/\Z^{4} \to \R^{3}/\Z^{3}
\]
which adds the first two components. We shall consider only cases such
as this when discussing reduced instanton homology in this paper, in
order to have the components of $K$ on an equal footing. We then
define $\Inat(K; \Gamma_{\s})$ using the Morse homology of the
peturbed Chern-Simons functional on $\bonf^{\natural}(K)$, with
coefficients in $\Gamma_{\s}$.

This reduced instanton homology is functorial for ``based'' cobordisms
of links. Given links $(Y_{0}, K_{0})$ and $(Y_{1}, K_{1})$, with
framed base-points $p_{0}$ and $p_{1}$ on the links, the appropriate
morphism is given by a
cobordism of pairs, $(X,S)$ together with an arc $\gamma\subset S$
joining the base-points and a framing $(\epsilon_{1}, \epsilon_{2},
\epsilon_{3})$ of the normal to $\gamma$ in $X$ such that
$\epsilon_{1}$ is tangent to $S$. Equivalently, we can think of an
embedding of $[0,1]\times B^{3}$ in $X$ which intersects $S$ in the
image of the standard $[0,1]\times B^{1}$. Given such data, we can
perform the bigon addition (Figure~\ref{fig:reduced}) in a
one-parameter family along the image of $[0,1]\times B^{3}$, to obtain
an embedded foam $S^{\natural}$ with boundary $K^{\natural}_{0} \cup
K^{\natural}_{1}$. The foam gives rise to homomorphisms
\[
    \Inat(X,S; \Gamma_{\s}) : \Inat(Y_{0}, K_{0}; \Gamma_{\s}) \to
     \Inat(Y_{1}, K_{1}; \Gamma_{\s}) 
 \]
where the matrix entries at the chain level are given by the same
formulae \eqref{eq:T-cobordism} as in the non-reduced case, with the
$\nu_{i}$ being the curvature integrals over the facets of $S^{\natural}$.

\subsection{The K\"unneth theorem for reduced homology}

Given links $K_{1}$ and $K_{2}$, each with a framed basepoint, there
is a natural construction of the connected sum $K_{1}\csum K_{2}$,
also as a link with framed base-point. To spell this out, let
$(\epsilon_{1}, \epsilon_{2}, \epsilon_{3})$ be the framing at the
base-point of $K_{1}$, with $\epsilon_{1}$ pointing along the
knot. Using the framing, parameterize a standard ball $B_{3,1}$ around
the base-point. Construct $B_{3,2}$ similarly. Remove the interiors
and form the connected sum by identifying the $2$-sphere $\partial B_{3,1}$ with
$\partial B_{3,2}$ using the orientation-reversing diffeomorphism
given by reflection in the $\epsilon_{1}$ direction. Take the
base-point on the new link to be the image of the point $\epsilon_{1}$
on $\partial B_{3,1}$.

The construction of $K_{1}\csum K_{2}$ from the two framed knots is
functorial. That is, given cobordisms $S_{i}$ from $K_{i}'$ to $K_{i}$
for $i=1,2$, and given framed arcs $\gamma=(\gamma_{1}, \gamma_{2})$
joining the framed basepoints, we can form a cobordism
\[
             S_{1} \csum_{\gamma} S_{2}
\]    
from $K'_{1}\csum K_{2}'$ to $K_{1}\csum K_{2}$ by performing the
connected-sum construction in an interval family. The reduced
instanton homology for a connected sum of framed knots is described as
a tensor product by a K\"unneth theorem:

\begin{proposition}\label{prop:Kunneth-red}
    Let $(C_{1}, \partial)$ and $(C_{2},\partial)$ be the differential
    $\cS$-modules arising from the Floer complexes for the
    homology groups $\Inat(K_{1};\Gamma_{\s})$ and
    $\Inat(K_{2};\Gamma_{\s})$. Then the Floer complex for
    $K_{1}\csum K_{2}$ is chain-homotopy equivalent to the tensor
    product $C_{1}\otimes_{\cS} C_{2}$. In particular, if\/
    $\cS$ is a principal ideal domain, then there is a split
    exact sequence of\/ $\cS$-modules,
    \begin{equation}\label{eq:Kunneth}
                 \begin{aligned}                     
                  0 \longrightarrow \Inat(K_{1};\Gamma_{\s})
                  \otimes_{\cS}
                  \Inat(K_{2};\Gamma_{\s}) \longrightarrow
                  \null &\Inat(K_{1}\csum K_{2};\Gamma_{\s})\\
                  &\longrightarrow
                  \mathrm{Tor}^{\cS}_{1} \bigl(\Inat(K_{1};\Gamma_{\s}),
                  \Inat(K_{2};\Gamma_{\s})\bigr) \longrightarrow 0.
                  \end{aligned}
              \end{equation}
     The exact sequence, but not the splitting, is natural with
     respect to the maps induced by cobordisms $S_{1}$, $S_{2}$ and
     $S_{1}\csum_{\gamma}S_{2}$ as constructed above.         
\end{proposition}

\begin{figure}
    \begin{center}
        \includegraphics[scale=0.5]{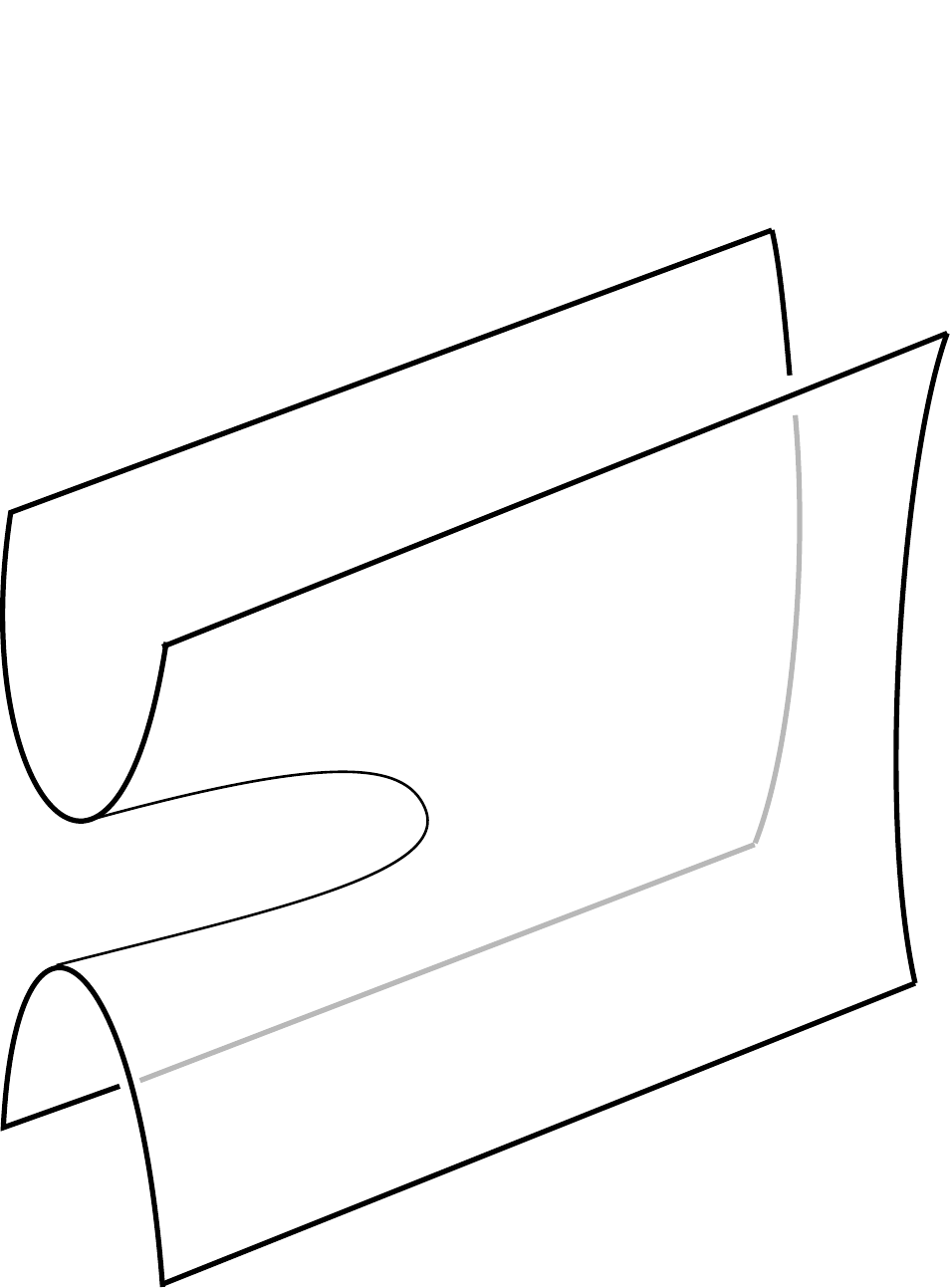}
    \end{center}
    \caption{\label{fig:Excision}
    The cobordism $U$ from two copies of an interval $I$ to another two.}
\end{figure}

\begin{proof}
    This a standard application of excision, as we now describe. The
    symmetries in the argument are more apparent in a more general
    version, so we consider four pairs $(Y_{i}, K_{i})$,
    $k=1,\dots,4$, where each $K_{i}$ is a based link. For each $i\ne
    j$, there is a connect-sum of pairs,
    \[
        (Y_{ij},K_{ij}) =   (Y_{i} , K_{i}) \csum (Y_{j}, K_{j}),
      \]
    where the 3-manifolds and the links are both summed at the
    base-points. Let $\Cnat_{ij}$ denote the chain group of free
    $\cS$-modules arising as the instanton Floer complex for this
    connected sum of based pairs, with coefficients in the local
    system $\Gamma_{\s}$. The more general statement is then that
    there is a chain-homotopy equivalence,
    \begin{equation}\label{eq:Cnat-1234}
              \Cnat_{12} \otimes_{\cS} \Cnat_{34} \simeq \Cnat_{13}
              \otimes_{\cS} \Cnat_{24},
          \end{equation}
    and that the resulting maps on homology are natural for
    cobordisms. The statement of the original proposition arises as a
    special case, when each $Y_{i}$ is $S^{3}$, and $K_{3}$ and
    $K_{4}$ are both the unknot, so that $(Y_{13}, K_{13})=(Y_{1}, K_{1})$,
     $(Y_{24}, K_{24})=(Y_{2}, K_{2})$,
     and $\Cnat_{34}=\cS$.

     We now recall Floer's excision argument, particularly in the
     versions described in \cite[Proposition~4.2]{KM-jsharp} and
     \cite[Proposition~3.3]{KM-deformation}. Let $U$ be as in
     Figure~\ref{fig:Excision}, a 2-dimensional cobordism from the
     $1$-dimensional manifold-with-boundary $I\cup
     I$ to $I\cup I$. Take the product with $S^{2}$ to obtain a
     cobordism from $I\times S^{2}\cup I\times S^{2}$ to $I\times
     S^{2} \cup I\times S^{2}$. Then attach four copies of
     $[0,1]\times B^{3}$ to obtain a cobordism $W$ from $S^{3}\cup S^{3}$
     to $S^{3}\cup S^{3}$. Inside $W$ there is an embedded foam,
     $\Phi$, formed from three copies of $U$. The pair $(W,\Phi)$ is a
     cobordism
     \[
         (S^{3},\theta) \cup (S^{3},\theta)\;\; \text{to}\;\;
          (S^{3},\theta) \cup (S^{3},\theta).
      \]
     On one facet of the three facets of $\Phi$, let $\gamma_{i}$,
     $i=1,\dots,4$, be four arcs as shown in
     Figure~\ref{fig:Phi-arcs}. A regular neighborhood of $\gamma_{i}$
     in $(W,\Phi)$ is a standard pair $[0,1]\times (B^{3}, B^{1})$
     along which we form a sum with $[0,1]\times (Y_{i}, K_{i})$. The
     result is a cobordism of pairs, $(X,\Psi)$ from
     \[
         (Y_{12}, K_{12}^{\natural}) \cup
         (Y_{34}, K_{34}^{\natural}) \;\; \text{to}\;\;
         (Y_{13}, K_{13}^{\natural}) \cup
         (Y_{24}, K_{24}^{\natural}).         
     \]
     As in the proof of  \cite[Proposition~3.3]{KM-deformation}, this
     cobordism of pairs gives rise to a map on the instanton chain
     complexes with local coefficients, in this case a chain map
     \[
              \Cnat_{12} \otimes_{\cS} \Cnat_{34} \to \Cnat_{13}
              \otimes_{\cS} \Cnat_{24}.
     \]
     By the same construction, a map in the other direction is
     constructed. The fact that the composite of the two, in either
     order, is chain-homotopic to the identity is proved by the usual
     argument, as in \cite{KM-deformation} for example.      
\end{proof}

\begin{figure}
    \begin{center}
        \includegraphics[scale=0.5]{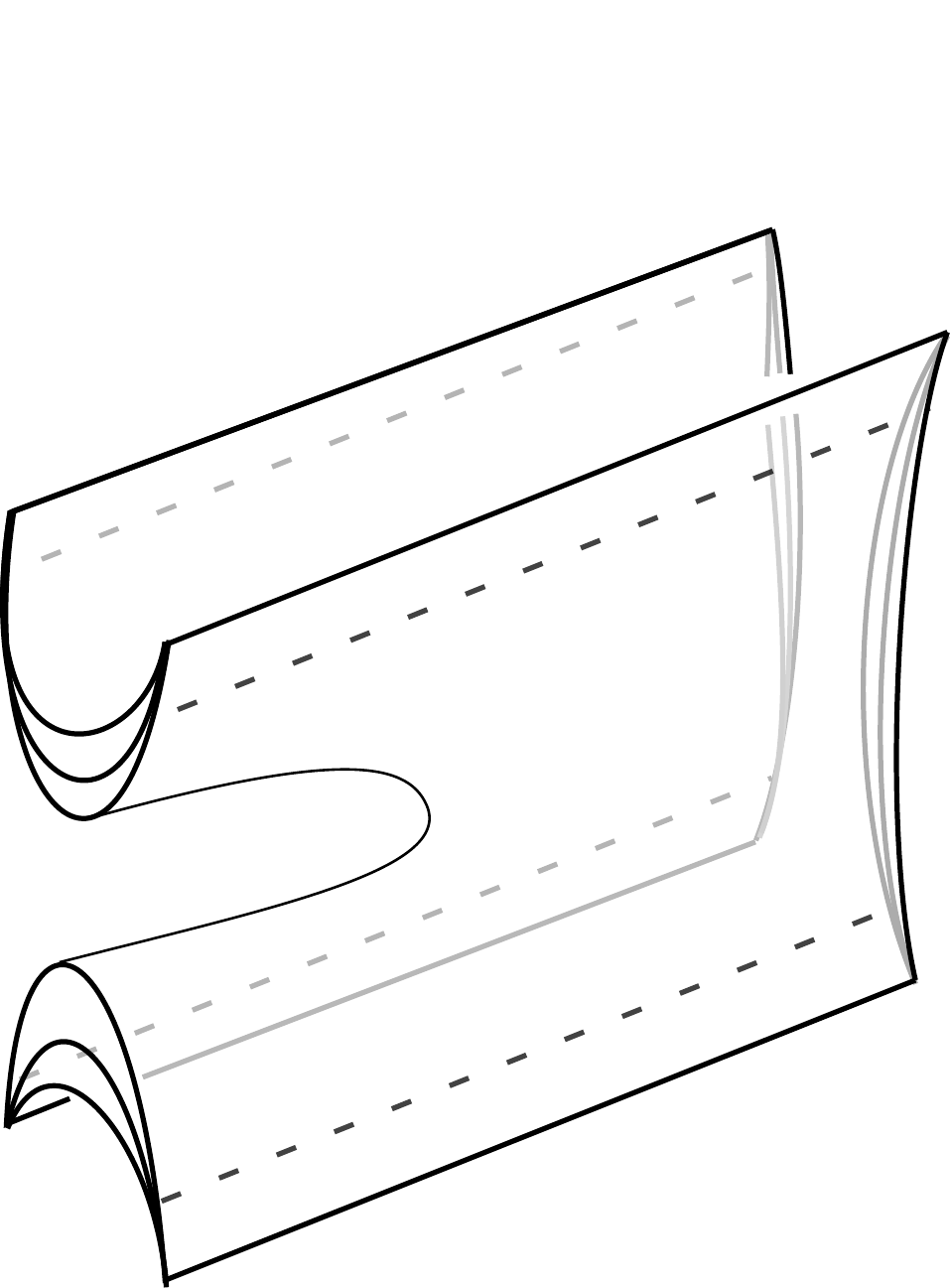}
    \end{center}
    \caption{\label{fig:Phi-arcs}
    The foam $\Phi\subset W$ and the four arcs along which the pairs
    $[0,1]\times (Y_{i},K_{i})$ are summed.}
\end{figure}

\section{Operators on  \texorpdfstring{$\Isharp(K;\Gamma)$}{I(K;Gamma)} }

We continue with the notation of the previous sections. We write
$Y^{o}\subset Y$ for the complement of a ball around a basepoint in
$Y$, and we consider a link $K\subset Y^{o}$, along with the union $K^{\sharp} =
K\cup \theta$ in $Y$. The space of connections $
\bonf^{\sharp}(K)_{\omega}$ carries a system of local coefficients
$\Gamma_{\s}$,
as in Notation~\ref{notation:Gamma}), 
and
$\Isharp(K ; \Gamma_{\s})_{\omega}$ is the instanton homology for the perturbed
Chern-Simons functional on $\bonf^{\sharp}(K)_{\omega}$, with coefficients in 
the local system.

\subsection{Operators from characteristic classes of the basepoint bundle.} 

Given an point $y$ in the smooth part of the orbifold $(Y, K)$,
there is a basepoint $\SO(3)$-bundle $\mathbb{E}_{y}$ on the
configuration space, with Stiefel Whitney classes \[ \mathsf{w}_{1} ,
\mathsf{w}_{2}, \mathsf{w}_{3} \in H^{*}(\bonf^{\sharp}(K)_{\omega}; \Z/2).\] (See \cite[section
4]{KM-deformation}.) For the instanton homology $\Isharp(K ;\Gamma_{\s})$, these characteristic
classes give rise to linear operators,
\[
         w_{i} : \Isharp(K ; \Gamma_{\s})_{\omega} \to  \Isharp(K ; \Gamma_{\s})_{\omega}.
\]
The definition of these for the similar case of $\Jsharp(K ;
\Gamma_{\theta})$ is presented in
\cite{KM-deformation} and needs essentially no change. As in \cite{KM-deformation}, we
have: 

\begin{lemma}\label{lemma:w-relations}
    On $\Isharp(K;\Gamma_{\s})_{\omega}$ the operators $w_{1}$ and $w_{3}$ are zero,
    while $w_{2}$ is multiplication by $\s(P)\in \cS$, where $P\in \cR$ is the
    element given by the expression in \eqref{eq:PQ-formulae}. \qed
\end{lemma}

\begin{figure}
    \begin{center}
        \includegraphics[scale=0.5]{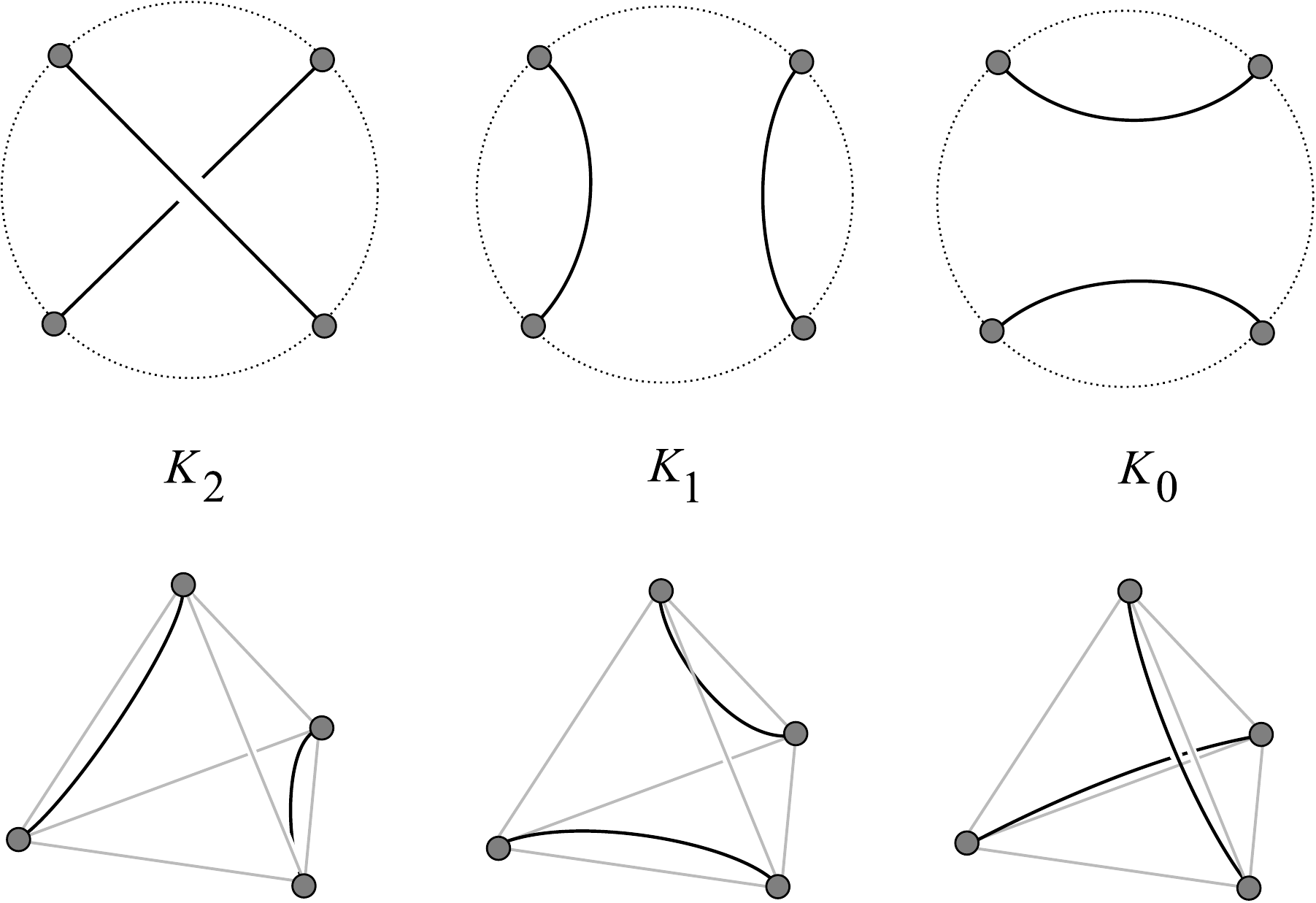}
    \end{center}
    \caption{\label{fig:Tetrahedra-skein}
    Links or webs $K_{2}$, $K_{1}$ and $K_{0}$ differing by the unoriented
    skein moves, in two different views.}
\end{figure}

\subsection{A two-dimensional cohomology class}
\label{subsec:Lambda-class}

Let $\arc$ be an arc in $Y$ with endpoints $\{p,q\}$ on $K\cup \theta$. The
interesting case will be when $p$ and $q$ lie on different
components, for example on $K$ and $\theta$ respectively. We require
that $p$ and $q$ lie on the interior of edges, not at the vertices of
the graph. We also require that $p$ and $q$ do not lie at endpoints of
$\omega$ (if any). There is a universal $\R^{3}$ bundle
\[
       \uE \to \arc \times \bonf^{\sharp}(K)_{\omega}.
\]
The restriction, $\uE_{p}$, of $\uE$ to the endpoint $\{p\} \times
\bonf^{\sharp}(K)_{\omega}$ carries an involution on the $\R^{3}$
fibers, because of the $\Z/2$
stabilizer at this singular point of the orbifold, so $\uE_{p}$
contains a distinguished real line subbundle, the $+1$ eigenspace of
the involution: \[  \mathbb{L}_{p} \subset \uE_{p} \to \bonf^{\sharp}(K)_{\omega}.\]

This line bundle is trivial. Indeed, we have:

\begin{lemma}
    Given the condition that $p$ is not a boundary point of $\omega$,
    a choice of orientation $o_{p}$ for\/ $\mathbb{L}_{p}$ is determined
    by an orientation of $K$ at $p$. If\/ $p_{1}$ and $p_{2}$ lie either
    side of a single endpoint of $\omega$ on $K$, and if $K$ is given
    the same orientation at $p_{1}$ and $p_{2}$, then the
    corresponding orientations $o_{p_{1}}$, $o_{p_{2}}$ of\/
    $\mathbb{L}_{p_{1}} \cong \mathbb{L}_{p_{2}}$ are opposite.
\end{lemma}

\begin{proof}
We use the characterization of
$\bonf^{\sharp}(K)_{\omega}$ in Lemma~\ref{lemma:bonf}.  The connection
$[A]\in \bonf^{\sharp}(K)_{\omega}$ has a preferred lift to an
$\SU(2)$ connection $\hat A$ in $U_{p}\setminus K$ for some neighborhood
$U_{p}$ of $p$ in $Y$. After orienting (the normal bundle to) $K$, we
can consider the limiting holonomy of $\hat A$ around small circles
linking $p\in K$, which is an element of order $4$ in $\SU(2)$. The
2-sphere which parametrizes elements of order $4$ is identified with
the unit sphere in the $\R^{3}$ bundle, and under this identification the limiting holonomy is
an element of $\mathbb{L}_{p}$
\end{proof}

With the lemma in mind, we introduce the following notation.

\begin{definition}\label{def:dot}
    A \emph{dot} on $K$ is a chosen point $p$ on $K$, not a boundary
    point of $\omega$, together with a choice of orientation $o_{p}$
    for the line bundle $\mathbb{L}_{p}\to
    \bonf^{\sharp}(K)_{\omega}$. We may omit explicit mention of
    $o_{p}$, and simply refer to $p$ as a dot. If $p$ is a dot, we
    write $\bar{p}$ for dot with the same underlying point and the
    opposite orientation for the line bundle. We note that a choice of
    orientation of $\mathbb{L}_{p}$ is equivalent to a choice of
    orientation of $K$ near $p$.
\end{definition}

Suppose now that $p$ and $q$ are dots, and let us return to the arc
$a$ joining them as introduced above. The dot $p$ determines a
distinguished section of $\mathbb{L}_{p}$ and hence
a distinguished section $i_{p}$ of the $\R^{3}$ bundle
$\uE_{p}$. Similarly, using the dot $q$, 
we obtain a distinguished section $i_{q}$. Changing the sign
of the second one, we obtain a distinguished section
\[
         I = (i_{p} , -i_{q})
\]
of the restriction of $\uE$ to $\partial a \times
\bonf^{\sharp}(K)$. 

The distinguished section on the boundary allows us to define a relative
Euler class; or a top Stiefel-Whitney class
\[
\begin{aligned}
    \mathbb{w}_{3} &= w_{3}(\uE , I) \\ & \in H^{3}\left(
        (\arc, \partial \arc) \times \bonf^{\sharp}(K) ; \F_{2}
    \right).
\end{aligned}
\]
We now take the slant product with the relative fundamental class of
the arc to obtain a class on $\bonf^{\sharp}(K)_{\omega}$:

\begin{definition}
For an arc $a$ as above whose endpoints $p$ and $q$ are dots, 
we define a 2-dimensional cohomology class with $\F_{2}$ coefficients
on $\bonf^{\sharp}(K)_{\omega}$ as
    \begin{equation}\label{eq:lambda-class}
        \begin{aligned}
            \lambda &= \mathbb{w}_{3} / [\arc, \partial \arc] \\
            &\in H^{2}\left( \bonf^{\sharp}(K) ; \F_{2} \right).
        \end{aligned}
    \end{equation}
\end{definition}

The next lemma shows that the arc $a$ itself plays only an auxiliary
role in this construction.

\begin{lemma}
    In the above construction, the class $\lambda$ depends only on the
    dots $p$, $q$. It does not otherwise depend on the arc $\arc$.
\end{lemma}

\begin{proof}
     Consider the universal $\R^{3}$ bundle
    \[
            \uE \to (Y\setminus K)  \times \bonf^{\sharp}(K)_{\omega}.
    \]
    The assertion to be proved is equivalent to the statement that
    \[
    \begin{aligned}
        w_{3}(\uE) / [b] & = 0 \\
        &\in H^{2}\left( \bonf^{\sharp}(K)_{\omega} ; \F_{2} \right).
    \end{aligned}
     \]
   for any 1-cycle $b$ in $Y\setminus K$. The bundle has
   trivial $w_{2}$ on $Y\setminus K$, so each irreducible connection
   lifts to an $\SU(2)$ connection with stabilizer $\pm 1$.  This
   means that $w_{2}(\uE)$ can be represented by a $2$-cocycle which
   is pulled back from $ \bonf^{\sharp}(K)_{\omega} $. The class $w_{3}$ is
   obtained by applying a Bockstein homomorphism to $w_{2}$. So the
   class $w_{3}(\uE)$ is also pulled back from $
   \bonf^{\sharp}(K)_{\omega}$. It follows that $w_{3}(\uE) / [b]$ is zero,
   for all $1$-cycles $b$ in $Y\setminus K$ (and incidentally all $2$-cycles also). 
\end{proof}

The lemma allows us to write the class as a function of the two
dots,
\begin{equation}\label{eq:lambda-class-pq}
             \lambda_{pq} \in  H^{2}\left( \bonf^{\sharp}(K) ; \F_{2} \right),
\end{equation}

The next lemma asks how $\lambda_{pq}$ changes if we change the
orientation $o_{q}$ at one endpoint: that is we replace $q$ by
$\bar{q}$. We introduce the following notation: we write
\[
             \lambda'_{pq} = \lambda_{p\bar{q}}.
\]

\begin{lemma}\label{lemma:lambda-relations-cohomology}
    Let  $p$ and $q$ be dots, and let $\lambda_{pq}$ and $\lambda'_{pq}$ be the
    resulting classes, as above. Then these classes satisfy the relations:
    \begin{equation}\label{eq:lambda-relations}
        \begin{aligned}
            \lambda_{pq} + \lambda'_{pq} &= w_{2}(\uE_{q}) \\
            \lambda_{pq} \lambda'_{pq} &= 0
        \end{aligned}
     \end{equation}
\end{lemma}

\begin{proof}
Because of the independence of the choice of arc, and the way the
signs are used in the definition of $I$ above, the first relation is equivalent to
saying
\[
          \lambda_{qq} = w_{2}(\uE_{q}),
\]    
where the left-hand side can be computed using the constant arc from
$q$ to $q$.

The general statement at the level of characteristic classes is the following.
Suppose we have an $\R^{3}$ bundle $E\to T$ with a section
$i_{0}$. Consider the pull-back $\pi^{*}(E)$ to $[0,1]\times T$ with a section
$I$ which is equal to $i_{0}$ on $\{0\}\times T$ and $-i_{0}$ on
$\{1\}\times T$. Then the result of slanting $w_{3}(\pi^{*}(E),
I)$ with the fundamental class of $[0,1]$ is $w_{2}(E)$:
\[
        w_{3}(\pi^{*}(E), I)/ [0,1] = w_{2}(E).
\]
This can be verified by pulling back $E$ to $S^{1}\times T$ and
tensoring by the M\"obius bundle $\mu$ on $S^{1}$, in which case the
assertion is:
\[
      w_{3}(\pi^{*}(E)\otimes \mu )/ [S^{1}] = w_{2}(E).
\]
In this form, the verification is straightforward, using the splitting
principle. This completes the proof of the first relation.

To set the second relation in a more general context, consider again
an $\R^{3}$ bundle $E\to T$ with two non-vanishing sections $i_{0}$
and $i_{1}$. Let $I$ be a path of sections, from $i_{0}$ to $-i_{1}$,
through sections which may vanish: we  take explicitly
\[
         I(t) = (1-t) i_{0} - t i_{1}.
     \]
     Similarly, let $I'(s)$ be the path from $i_{0}$ to $i_{1}$ given by
     \[
         I'(s) = (1-s) i_{0} + s i_{1}.
     \]
We have cohomology classes
by \[
    \begin{aligned}
        \lambda&=w_{3}(E,I)/[0,1]\\ & \qquad \qquad \in H^{2}(T;\F_{2}) \\
        \lambda'&= w_{3}(E, I')/[0,1]\\ & \qquad \qquad  \in H^{2}(T;\F_{2}),
    \end{aligned}
\]
where we now interpret $I$ and $I'$ as sections on $[0,1]\times T$
that are non-zero at the boundaries. To show that $\lambda\lambda'=0$,
it is sufficient to show that there is no $(t,s)$ in the interior of $[0,1]\times[0,1]$
for which the sections $I(t)$ and $I'(s)$ have a common zero in $T$. A
necessary condition for a common zero is that the determinant of the matrix
\[
    \begin{pmatrix}
        (1-t) & -t \\
        (1-s) & s
    \end{pmatrix}
\]
is zero. But the determinant
is $s+t-2st$ which is strictly positive on the interior of
$[0,1]\times [0,1]$.
The result follows.
\end{proof}

\begin{corollary}
    The class $\lambda_{pq}$ satisfies the relation
    \[
                      \lambda^{2}_{pq} + w_{2}(\uE_{q}) \lambda_{pq} = 0.
    \]    
\end{corollary}

\begin{proof}
    This is an immediate corollary of the two relations in the lemma.
\end{proof}

\subsection{Operators from the two-dimensional classes}

In the usual way, and following the exposition in \cite{KM-deformation}, the cohomology class $\lambda_{pq}$ gives rise
to an operator
\begin{equation}\label{eq:Lambda-operator}
        \Lambda_{pq} : \Isharp(Y, K ; \Gamma_{\s})_{\omega} \to  \Isharp(Y, K ; \Gamma_{\s})_{\omega}.
\end{equation}
In a little more detail, let $a$ be the chosen arc joining the two dots,
regarded as subset of the cylinder $\check X = \R\times \check
Y$, in the slice where the $\R$ coordinate is zero.
Following  \cite[section 4.3]{KM-deformation}, let
$Z\subset \check X$ be a subset of $\check X$ which includes a
neighborhood of $a$ and such that the restriction map
\[
           H^{1}(Y\setminus K; \F_{2}) \to  H^{1}(Z\setminus K; \F_{2}) 
\]
is injective. The latter condition means there is a well-defined
restriction map for marked connections,
\[
          M(\alpha, \beta) \to \bonf^{*}(Z ; \mu_{Z})
\]
where $\mu_{Z}$ is the intersection with $Z$ of the marking data
$\R\times \mu_{\omega}$. Because $Z$ contains a neighborhood of $a$, the
class $\lambda_{pq}$ can be defined on $\bonf^{*}(Z ; \mu_{Z})$, where
it is dual to a stratified codimension-2 subvariety $V$. The matrix
entries of $\Lambda_{pq}$ at the chain level are defined by counting
points of the intersections
\[
        M(\alpha, \beta) \cap V
\]
and weighting them using the local system. As in
\cite{KM-deformation}, the necessary compactness results hold because
the cohomology class has dimension $2$ and $Z$ can be chosen so that
it meets the foam only in the faces (at neighborhoods of $p$ and $q$).

Using notation that is parallel to the notation for $\lambda_{pq}$ and
$\lambda'_{pq}$, we write $\Lambda'_{pq}$ for the operator of the same
form as \eqref{eq:Lambda-operator}, but using $\bar{q}$ in place of
$q$. 
The relations in
Lemma~\ref{lemma:lambda-relations-cohomology}, satisfied by
$\lambda_{pq}$ and $\lambda'_{pq}$, give rise to relations satisfied
by the corresponding operators on Floer homology.

\begin{lemma}\label{lemma:Lambda-relations}
    Let  $\Lambda_{pq}$ and $\Lambda'_{pq}$ be the
    operators on $\Isharp(K;\Gamma_{\s})_{\omega}$ arising from a pair
    of dots $\{p,q\}$ as above.
   Then these operators satisfy the relations:
    \begin{equation}\label{eq:Lambda-relations}
        \begin{aligned}
            \Lambda_{pq} + \Lambda'_{pq} &= \s(P) \\
            \Lambda_{pq} \Lambda'_{pq} &= \s(Q_{pq}),
        \end{aligned}
     \end{equation}   
    where $P$ is the element of\/ $\cR$ given by \eqref{eq:PQ-formulae},
    and
    \[
             Q_{pq} = (T_{m_{p}}^{2} + T_{m_{p}}^{-2}) + 
                            (T_{m_{q}}^{2} + T_{m_{q}}^{-2}).
     \]
   In the above formula, $T_{m_{p}}$ and $T_{m_{q}}$ are the variables
   from $\{T_{0}, T_{1}, T_{2}, T_{3}\}$ associated to the edges of\/ $K
   \cup \theta$ on which $p$ and $q$ lie. Thus, $m_{p}=1, 2$ or $3$ if
   $p$ lies on the edge $e_{1}$, $e_2$ or $e_{3}$ of $\theta$, and
   $m_{p}=0$ if $p$ lies on $K$.
\end{lemma}

\begin{proof}
    By an excision argument \cite{KM-unknot}, it is sufficient to prove this in the
    case that $\omega$ is empty. It is then sufficient to consider the
    case that the base-change $\s$ is the identity.

    The first of the two relations follows from the corresponding
    formula for $\lambda_{pq} + \lambda'_{pq}$ in
    Lemma~\ref{lemma:lambda-relations-cohomology}, together with the
    formula  $w_{2}=P$ from Lemma~\ref{lemma:w-relations}.

    The second relation, for the product, is more subtle, because it
    involves a 4-dimensional moduli space, and there is a contribution
    from codimension-4 bubbling which may occur at the endpoints $p$
    and $q$ of the arc $a \subset Z$.

    As in \cite{Yi} and \cite{KM-deformation}, the contribution from
    the bubbles at $p$ and $q$ are universal quantities, so that the
    relation for the product has the general shape
    \[
                   \Lambda_{pq} \Lambda'_{pq} = F(T_{m_{p}}) +
                   F(T_{m_{q}})
      \]
    where $F$ is universal and is a finite Laurent series in one variable. To
    compute $F$, we take as a special case that situation that $K$ is
    empty
    and $p$ and $q$ lie on the edges $e_{2}$ and $e_{1}$ of
    $\theta$, respectively.  In the
    ring $\F_{2}[T_{1}^{\pm 1}, T_{2}^{\pm 1}, T_{3}^{\pm 1}]$ then, we have elements
    $\Lambda$ and $\Lambda'$ with relations
    \[
    \begin{aligned}
        \Lambda  + \Lambda' &= P \\
        &= T_{1}T_{2}T_{3} + T_{1}T_{2}^{-1}T_{3}^{-1} +
        T_{2}T_{3}^{-1}T_{1}^{-1} + T_{3}T_{1}^{-1}T_{2}^{-1}
    \end{aligned}
   \] 
    and
   \[
              \Lambda \Lambda' = F(T_{2}) + F(T_{1}).
    \] 
     The only way to solve the constraint that $\Lambda\Lambda'$ is
     function of $T_{1}$ and $T_{2}$ only is to have the general shape
     \[
     \begin{aligned}
         \Lambda  & = T_{3}^{a} G(T_{1}, T_{2}) \\
         \Lambda'& = T_{3}^{-a} H(T_{1}, T_{2}),
     \end{aligned}
\]
    for some Laurent polynomials $G$ and $H$ in two variables.
    The shape of the formula for $\Lambda+\Lambda'$ tells us that $a$
    must be $\pm 1$ and that $\Lambda$ and $\Lambda'$ must consist of
    the corresponding monomials from the formula for $P$. Thus
     \begin{equation}\label{eq:local-lambdas}
     \begin{aligned}
         \Lambda  & = T_{3}(T_{1}T_{2} + T_{1}^{-1}T_{2}^{-1}) \\
         \Lambda'& = T_{3}^{-1}(T_{1}^{-1}T_{2} + T_{1}T_{2}^{-1})
     \end{aligned}
     \end{equation} 
   or vice versa. Either way, we have $F(T)= T^{2} + T^{-2}$.
\end{proof}

As with the cohomology classes themselves, we have an immediate
corollary of the lemma, for the operator $\Lambda_{pq}$:

\begin{corollary}\label{cor:Lambda-reln}
    The operator $\Lambda_{pq}$ satisfies the relation
    \[
                      \Lambda^{2}_{pq} +\s( P) \Lambda_{pq} + \s(Q_{pq}) = 0.
    \]    
\end{corollary}

To create an operator that treats the three edges of $\theta$
symmetrically, we make the following definition.

\begin{definition}\label{def:Lambda3}
    Fix once and for all a dot $p_{m}$ on each edge of  $e_{m}$ of the
    $\theta$, for $m=1,2,3$. Then, given a dot
    $q$ on the link $K$, we define
    \[
        \Lambda_{q} = \Lambda_{p_{1}q} + \Lambda_{p_{2}q} +
        \Lambda_{p_{3}q} ,
    \]
    with $\Lambda'_{q}$ defined similarly.
\end{definition}

\begin{corollary}\label{cor:norm-Lambda-reln}
    The operator $\Lambda_{q}$ above satisfies the relation
    \[
                      \Lambda^{2}_{q} + \s(P) \Lambda_{q} +\s( Q) = 0,
    \]     
    where $P$ and $Q$ are given by the formulae in
    \eqref{eq:PQ-formulae}. Furthermore
    $\Lambda_{q} + \Lambda'_{q}= \s(P)$.
\end{corollary}

\begin{proof}
    We are in characteristic $2$, where squaring is linear. The $Q$
    that appears in the quadratic relations is now the sum of the
    terms $Q_{p_{m} q}$ for $m=1,2,3$. 
\end{proof}

\subsection{Surfaces with dots}
\label{subsec:surfaces-with-dots}

The operators $\Lambda_{q}$ on $\Isharp(K;\Gamma_{\s})_{\omega}$ that we have
defined can be combined -- in the usual way -- 
with the functorial maps obtained from cobordisms $S$ between knots and
links. Thus, suppose we are given a cobordism $(X,S)$ from
$(Y_{0},K_{0})$ to $(Y_{1}, K_{1})$ as in \eqref{eq:functorial}, and
let $q$ be  a  dot on $S$. As before, this means a point with a choice
of orientation $o_{q}$ of the line bundle $\mathbb{L}_{q}$. We then obtain a map
  \begin{equation}\label{eq:functorial-dot}
     \Isharp((S;q) ; \Gamma_{\s})_{\omega} : \Isharp(  K_{0}; \Gamma_{\s})_{\omega_{0}} \to \Isharp( K_{1}; \Gamma_{\s})_{\omega_{1}}.
\end{equation}
If $q$ can be joined by a path on $S$ to a point $q_{0}\in K_{0}$
(respectively, a point $q_{1}\in K_{1}$),
then this map is equal to the composite,
\[
                \Isharp(S ; \Gamma_{\s})_{\omega} \circ \Lambda_{q_{0}},
\]
respectively
\[
               \Lambda_{q_{1}} \circ  \Isharp(S ; \Gamma_{\s})_{\omega}.
\]
The functorial properties of $\Isharp$ extend to this larger category
in which the morphisms are ``cobordisms with dots''. We note that, as
with the case of dots on a link $K$, an orientation of the line bundle
$\mathbb{L}_{q}$ is equivalent to a choice of orientation for a
neighborhood of $q$ in $S$. So a dot can be regarded as point in $S$
together with an orientation of $T_{q}S$.

\section{Double points and handles}

As in section~\ref{subsec:functoriality}, let $Y_{0}$ and $Y_{1}$ be 3-manifolds with basepoints, containing
links $K_{0}$, $K_{1}$, and let  $\omega_{0}$, $\omega_{1}$ be
representatives for the Stiefel-Whitney class. We continue to use
$\Gamma_{\s}$ to denote $\Gamma \otimes_{\s} \cS$, and we consider again a
map
\begin{equation}\label{eq:S-map}
   \Isharp(X, S; \Gamma_{\s})_{\omega} : \Isharp(K_{0};
   \Gamma_{\s})_{\omega_{0}} \to \Isharp(K_{1} ; \Gamma_{\s})_{\omega_{1}}
\end{equation}
arising from a bifold cobordism $(X,S)$ and a choice of
Stiefel-Whitney class represented by a surface $\omega$ with
boundary. We continue to assume that $S$ is a surface rather than a more
general foam, and we recall that $\omega$ is allowed to have part of its
boundary on  $S$ if $\s(T_{0})=1$ . Implicit in our notation is an embedding of $[0,1]\times
B^{3}$ in $X$, containing the cylindrical foam $[0,1]\times \theta$, disjoint from
$S$ and $\omega$.
               
As in \cite{KM-unknot, KM-jsharp, KM-deformation}, we can consider how
the map $\Isharp(X,S ; \Gamma_{\s})_{\omega}$ changes when we modify
the topology of $S$ in standard
ways.

\subsection{Connect sum with \texorpdfstring{$\RP^{2}$}{RP2}}

In $S^{4}$, there are two standard copies of $\RP^{2}$ (see
\cite{KM-unknot} for example), which we call
$R_{+}$ and $R_{-}$. These have self-intersection numbers
\[
\begin{aligned}
    R_{+} \cdot R_{+} &= +2 \\
    R_{-} \cdot R_{-} &= -2 .
\end{aligned}
\]
From $(X,S)$ we can form a new cobordism as a connected sum,
\[
          (X,\tilde S)=  (X,S) \csum (S^{4}, R_{\pm }).
\]

\begin{lemma}
    In the case $\tilde S = S\csum R_{+}$, or $\tilde S = S \csum R_{-}$, we have
\[
        \Isharp(X,\tilde S ; \Gamma_{\s})_{\omega} = 0.
    \]
\end{lemma}

\noindent
We postpone the proof until after the statement of the next lemma.

If $\s(T_{0})=1$, then we can use more general representatives for classes
$w_{2}$. In particular, a circle representing the generator of
$H_{1}(R_{\pm})$ 
bounds a disk in the complement of $R_{\pm}$ in $S^{4}$. Let us write
$\pi$ for this disk. It represents a non-zero mod-2 class in the homology
of the complement. In the complement of $\tilde S = S \csum R_{\pm}$,
we can then use the Stiefel-Whitney class represented by $\omega+\pi$.

\begin{lemma}\label{lem:csum-R-theta}
    Suppose that $\s(T_{0})=1$ in the ring $\cS$. Then
    in the case that $\tilde S = S\csum R_{+}$, we have
\[
         \Isharp(X,\tilde S ; \Gamma_{\s})_{\omega+\pi} = \Isharp(X,S;\Gamma_{\s})_{\omega}.
\]
    In the case that $\tilde S = S\csum R_{-}$, we have
\[
         \Isharp(X,\tilde S ; \Gamma_{\s})_{\omega+\pi} = \s(P)\, \Isharp(X,S;\Gamma_{\s})_{\omega},
\]
where $P \in \cR$ is the element given by the formula \eqref{eq:PQ-formulae}.
\end{lemma}

\begin{proof}[Proof of the two lemmas.]
    There are four assertions altogether: two surfaces $R_{\pm}$, and
    two choices of Stiefel-Whitney class. In each case, we have a
    connected sum with $(S^{4}, R_{\pm})$ along $(S^{3}, S^{1})$. We
     apply the usual stretching argument, and we consider the possible
     weak limit on $(S^{4}, R_{\pm})$. The
    gluing parameter in the connected sum is $S^{1}$, so we will have
    non-zero contributions only when the weak limit on $(S^{4},
    R_{\pm})$ is an anti-self-dual connection with $S^{1}$
    stabilizer. However, there are no non-zero harmonic 2-forms on these
    orbifolds, so the only possibility is a flat connection. There is a
    unique flat $SO(3)$ bifold connection $[A_{\pm}]$ on $(S^{4}, R_{\pm})$ because the fundamental
    group of the complement is cyclic of order $2$. Its
    Stiefel-Whitney class is represented by $\pi$. This proves the
    first lemma: there is no contribution for the Stiefel-Whitney
    class $\omega$.

    The anti-self-dual connection $[A_{\pm}]$ is unobtructed in the
    case of $R_{+}$ and has a 2-dimensional obstruction space in the
    case of $R_{-}$, as explained in \cite[section 2.7]{KM-unknot}. So
    for $R_{+}$ we have
   \[
              \Isharp(X,S \csum R_{+} ; \Gamma_{\s})_{\omega+\pi} = \Isharp(X,S;\Gamma_{\s})_{\omega}.
   \]
  In the case of $R_{-}$, we can identify the 2-dimensional gluing
  obstruction with the 2-plane bundle $\eta$, and the effect of gluing
  is the same as cutting down by
  $w_{2}(\eta)$. Lemma~\ref{lemma:w-relations} tells us this is
  multiplication by $\s(P)$.
\end{proof}

\subsection{Connect sum with \texorpdfstring{$T^{2}$}{T2}}
\label{subsec:sum-T2}

Let $T$ be a standard unknotted torus in $\R^{3}$, and regard $T$ by
inclusion as a submanifold of $S^{4}$. We may form a connected sum
\[
                (X, \tilde S) = (X,S) \csum (S^{4}, T).
\]

\begin{lemma}\label{lem:torus-sum}
    When $(X,\tilde S)$ is formed from $(X,S)$ by a connected sum with
    the standard torus $T$ as above, we have
\[
    \Isharp(X,\tilde S ; \Gamma_{\s})_{\omega} = 
   \s( P)\, \Isharp(X,S;\Gamma_{\s})_{\omega}.
\]
\end{lemma}

\begin{proof}
    As with then previous two proofs, we are forming a sum along
    $(S^{3}, S^{1})$ and non-zero contributions arise from
    anti-self-dual bifold connections on $(S^{4}, T)$. These in turn
    come from reducible anti-self-dual connections on the branched
    double-cover, $S^{2}\times S^{2}$, which are invariant under the
    involution which fixes the torus $S^{1}\times S^{1}$. 
    As in the previous lemma, the contributions come
    only from
    the flat bifold bundle $[E,A]$ on $(S^{4}, T)$ corresponding to
    the trivial bundle $[\tilde E, \tilde A]$ on $S^{2}\times
    S^{2}$, because there are no non-zero harmonic 2-forms on the orbifold. 
    The obstruction space for $[E,A]$ is again
    two-dimensional, because it arises from
    $\mathcal{H}^{+}(S^{2}\times S^{2} ; \tilde E_{-})$, where $\tilde
    E_{-}$ is the two-dimensional summand of the trivial bundle on
    which the involution acts as $-1$. In the gluing, the obstruction
    bundle is again $\eta$, and the calculation is the same as the
    case $S\csum R_{-}$ from Lemma~\ref{lem:csum-R-theta}.
\end{proof}

\subsection{Double points and blowing up}
\label{subsec:double-points}

As in \cite{KM-singular} and \cite{KM-s-invariant}, we can extend the
definition of the maps $\Isharp(X,S;\Gamma_{\s})_{\omega}$ induced by cobordisms to
include also the case that $S$ is a normally immersed surface in
$X$. Our approach in the present paper is a slight variant of what was
done in the two previous cited papers: what we will do here is 
better-adapted to the case of an unoriented surface $S$.

So let $f : S \looparrowright X$ be an ``immersed cobordism'' from
$(Y_{0}, K_{0})$ to $(Y_{1}, K_{1})$. We always assume, as in
\cite{KM-singular}, that $f$ has only transverse double-points, and that
these are in the interior of $X$. That is, the surface is
\emph{normally immersed}. We also assume that the double
points do not lie on the surface $\omega$ which represents $w_{2}$.
We do not want to orient $S$, and we
therefore do not give a sign $\pm 1$ to the double-points of the
immersion. At a double-point $x\in f(S)$, we may choose the metric on
$X$ so that the two branches of the immersion have orthogonal tangent
planes, $\pi$ and $\pi'$ in $T_{x}X$. 
There are then exactly two  complex
structures $J$ and $-J$ on $T_{x}X$ such that:
\begin{enumerate}
\item the complex structure is compatible with the metric and
    orientation of $T_{x}X$;
\item $\pi$ and $\pi'$ are $J$-invariant;
\end{enumerate}
The \emph{blow-up} of $X$ at $x$ with respect to the complex
structures $J$ and $-J$ are canonically identified: in both blow-ups,
the exceptional set $\epsilon\subset \tilde X$ is the set of
$J$-invariant 2-planes in $T_{x}X$. When identified with $\CP^{1}$
however, the complex orientation of the exceptional set is different
in the two cases. The proper transform $\tilde f : \tilde S \looparrowright \tilde X$ has
one fewer double-point than $f$.

In the above situation,  we \emph{define} $\Isharp(X, S ; \Gamma)_{\omega}$ for
the immersed cobordism by requiring
\begin{equation}\label{eq:blow-up-rule}
         \Isharp(X, S ; \Gamma_{\s})_{\omega} = \Isharp(\tilde X, \tilde S
         ; \Gamma_{\s})_{\omega}
        +  \Isharp(\tilde X, \tilde S ; \Gamma_{\s})_{\omega+\epsilon}.
\end{equation}
On the right, we see the proper transform, equipped with two different
Stiefel-Whitney classes, differing by the exceptional set $\epsilon$
of the blow-up. By applying the definition to each double point in
turn, we arrive at a definition that reduces to the standard case of
\emph{embedded} cobordisms.

Before proceeding further, we make some remarks about this
definition. The proper transform is being used here to construct a
functor from a category in which the morphisms are immersed cobordisms
to one in which the morphisms are embedded cobordisms. In the previous
papers \cite{KM-singular, KM-s-invariant,Obstruction}, such a
construction was used with only the first of the two terms on the
right. The reason for using the two terms, involving both $\omega$ and
$\omega+\epsilon$, is to provide a deformation invariance that would
otherwise be absent, in the case that $\omega$ has boundary on $S$. To
understand this, consider a local model for a double point of $S$,
consisting of a pair of disks $D_{1}\cup D_{2}$ in the product
$D_{1}\times D_{2}$, and let coordinates be $(x_{i}, y_{i})$ be
standard coordinates on $D_{i}$, so that the disks meet at the
origin. Let $\omega$ be described in this neighborhood by
\[
       \omega = \{ \, y_{1} = 1/2, \; y_{2} = 0, \; x_{2}\ge 0 \, \}
\]
so that $\partial \omega$ is the line $y_{1}=1/2$ on the disk
$D_{1}$. Let $\omega'$ be obtained from $\omega$ by deforming $\omega$
in the this neighborhood in such a way that, (i) $\partial \omega'$ is the
line $y=-1/2$ on $D_{1}$; and (ii) $\omega'$ intersects $D_{2}$
transversely at a point. Let $(\tilde X, \tilde S$ be obtained by
blowing up the double-point $D_{1}\cap D_{2}$, and regard
$\omega$,$\omega'$ as lying in $\tilde X$. In this situation (assuming
that this local picture is just part of cobordism of pairs), we have
\[
          \Isharp(\tilde X, \tilde S
         ; \Gamma_{\s})_{\omega} \ne \Isharp(\tilde X, \tilde S
         ; \Gamma_{\s})_{\omega'}
\]
in general, because $\omega$ and $\omega'$ are representatives of
Stiefel-Whitney classes of different orbifold bundles. Instead, we
have
\[
          \Isharp(\tilde X, \tilde S
         ; \Gamma_{\s})_{\omega} = \Isharp(\tilde X, \tilde S
         ; \Gamma_{\s})_{\omega'+\epsilon},
\]
and similarly
\[
          \Isharp(\tilde X, \tilde S
         ; \Gamma_{\s})_{\omega+\epsilon} = \Isharp(\tilde X, \tilde S
         ; \Gamma_{\s})_{\omega'},
\]
So if we wish to define $\Isharp( X,  S
         ; \Gamma_{\s})_{\omega}$ when $S$ is normally immersed, and
         if we wish the result to be independent of the choice of
         $\omega$, in this way, we should take the two terms together,
         as we have done in \eqref{eq:blow-up-rule}.

With that said, if we impose the restriction that we consider only
$\omega$ without boundary along $S$, then we are free to modify the
definition of the functor: for any fixed choice of $\xi\in \cS$, we
can define a functor $\Isharp_{\xi}$ by leaving everything unchanged
except for the rule for dealing with double-points, where we
substitute the variant
\begin{equation}\label{eq:blow-up-rule-xi}
         \Isharp_{\xi}(X, S ; \Gamma_{\s})_{\omega} = \Isharp_{\xi}(\tilde X, \tilde S
         ; \Gamma_{\s})_{\omega}
        +  \xi \,\Isharp_{\xi}(\tilde X, \tilde S ; \Gamma_{\s})_{\omega+\epsilon}.
\end{equation}
For the rest of this paper, we shall remain with the more restricted
case $\xi=1$, with only occasional comments about the more general version.

\subsection{Twist moves and finger moves}

We now follow the strategy from \cite{KM-singular} to see how $\Isharp(X,
S ; \Gamma_{\s})$ changes when the immersion is changed in three standard
ways (introducing additional double-points).  These are the ``twist
move'', which comes in two oriented flavors, and the ``finger move''.

\begin{proposition}[\protect{\cf~\cite[Proposition 5.2]{KM-singular} and
      \cite[Proposition 3.1]{KM-s-invariant}}]
   \label{prop:twist-and-finger-formulae}
     Let $S^{*}$ be obtained from $S$ by either a positive
    twist move, or a finger move
    (introducing a canceling pair of double-points). Then we have,
    \[
              \Isharp(X, S^{*} ; \Gamma_{\s})_{\omega}  = \s(\V)\,
          \Isharp(X, S ; \Gamma_{\s})_{\omega} 
    \]
    where
     \[ \V  = P + T_{0}^{2} +
        T_{0}^{-2}.
   \] 
   For the negative twist move on the other hand, the map $\Isharp$ is
   unchanged:
   \[
              \Isharp(X, S^{*} ; \Gamma_{\s})_{\omega}  =
          \Isharp(X, S ; \Gamma_{\s})_{\omega} .
   \]
 \end{proposition}

\begin{remark}
    If we put $T_{1}=T_{2}=T_{3} = 1$, the formulae in the above
    proposition are essentially the same as those in \cite[Proposition
    5.2]{KM-singular}, but with the ``$t$'' from that earlier paper now
    replaced by $T_{0}^{2}$. The factor of $2$ in the exponent again
    arises because we have used the $\SO(3)$ connection rather than
    the $\SU(2)$ connection in defining the local system. Formulae of
    this sort go back to \cite{Obstruction}. The case of the finger
    move is also formally similar to crossing-change results in
    Heegaard-Floer homology \cite{OSS-Grid, Alishahi-Eftekhary} and in
    Bar-Natan homology \cite{Alishahi}.
\end{remark}

We prove the various parts of this proposition in the paragraphs. 
Our exposition describes just the case of $\Gamma$,
because the results are local, and the general $\Gamma_{\s}$ is obtained by base change.

\paragraph{Twist moves.}

We begin with the positive twist move. In this case, as explained in
\cite{KM-singular} and \cite{Obstruction}, the result of the positive
twist move followed by taking the proper transform in the blow-up is
to replace $(X,S)$ with
\[
          (X' , S') = (X, S) \csum (\bar{\CP}^{2} , C)
\]
where $C$ is a conic curve. According to our definition
\eqref{eq:blow-up-rule}, we must therefore compute
\[
\Isharp(X', S' ; \Gamma)_{\omega} + \Isharp(X', S' ;
\Gamma)_{\omega+\epsilon}.
\]
Once again, we compute by a connected sum argument. A dimension count
shows that the weak limit $[E,A]$ on $(\bar{\CP}^{2} , C)$ lies in a
moduli space of formal dimension $d_{0} \le -1$, which means that its
action $\kappa_{0}$ satisfies the bound $\kappa_{0} \le 1/4$. Since
the formal dimension is negative, the connection must be reducible,
either to $\pm 1$, to $\SO(2)$, or to $O(2)$. The double cover is
$S^{2}\times S^{2}$, with the involution $\tau(x,y) = (-y,-x)$. The
fixed-point set is the anti-diagonal $\Delta^{-}$. The pull-back
$[\tilde E, \tilde A]$ on $S^{2}\times S^{2}$ must be reducible,
either to $\SO(2)$ or the trivial group, so this $\SO(3)$ bundle has
the form
\[
          \R \oplus K
\]
where $e(K)$ can be taken to be $\tau$-invariant in the case that
$[E,A]$ reduces to $\pm 1$, and $\tau$-anti-invariant in the case that
$[E,A]$ reduces to $O(2)$. In the standard basis, $e(K)$ has the form
$(\delta, -\delta)$ or $(\delta, \delta)$ respectively. A class of the
second sort is not represented by an anti-self-dual form however. So
$e(K)=(\delta, -\delta)$ and $[E,A]$ reduces either to $\pm 1$ or to
$\SO(2)$. The bound on $\kappa_{0}$ means that $e(K)^{2} \ge -2$, so
$\delta^{2}\le 1$. 

If $\delta=0$, then $w_{2}(\tilde E) = 0$, which means that $w_{2}(E)
= \epsilon$ in the neighborhood of the blow-up. If $\delta=\pm 1$,
then $w_{2}(E)$ is zero in the neighborhood. The dimension count shows
that the two cases $\delta=\pm 1$ are unobstructed, and these
contribute the terms $T_{0}^{2} + T_{0}^{-2}$. (The calculation here
is just as in \cite{KM-singular}.) So we have
\[
           \Isharp(X', S' ;
\Gamma)_{\omega} = (T_{0}^{2} + T_{0}^{-2})  \Isharp(X, S ;
\Gamma)_{\omega}
\]

The case $\delta=0$ is the case of the flat bifold connection on
$(\bar\CP^{2}, C)$, and it contributes to the term $\Isharp( X', S' ;
\Gamma)_{\omega+\epsilon}$. The obstruction space is again
2-dimensional, and just as the case of a connected sum with either
$(S^{4}, R_{-})$ or $(S^{4}, T^{2})$, we obtain
\[
           \Isharp(X', S' ;
\Gamma)_{\omega+\epsilon} = P  \Isharp(X, S ;
\Gamma)_{\omega}
\]
This concludes the proof for the positive twist move. 

The negative twist move is straightforward. In this case we must
consider $(X', S')$ obtained from $(X,S)$ by forming the connect sum
with $(\bar \CP^{2}, \emptyset)$. The term with $\epsilon$ does not
contribute, and the term $ \Isharp(X', S' ;
\Gamma)_{\omega}$ is equal to $ \Isharp(X, S ;
\Gamma)_{\omega}$ as in \cite{KM-singular}.

\paragraph{Finger moves.}

Let $S^{*}$ be obtained from $S$ by a finger move, and let $S'$ be
obtained from $S^{*}$ by blowing up at the two double points and
taking the proper transform. Let $\epsilon_{1}$ and $\epsilon_{2}$ be
the exceptional sets of the two blow-ups. From the definition in
\eqref{eq:blow-up-rule}, we see that what we must compute is a sum of
four terms
\begin{equation}\label{eq:four-w2}
\I^{\sharp}(X', S' ; \Gamma)_{\omega} + 
\I^{\sharp}(X', S' ; \Gamma)_{\omega + \epsilon_{1}} + 
\I^{\sharp}(X', S' ; \Gamma)_{\omega + \epsilon_{2}} + 
\I^{\sharp}(X', S' ; \Gamma)_{\omega + \epsilon_{1}  + \epsilon_{2}},
\end{equation} 
and the desired answer is $\U \,\Isharp(X, S; \Gamma)_{\omega}$,
where $\U$ is as in part \ref{item:pos-and-finger} of the
Proposition.

To focus on the region where the change occurs, let us write
\[
        (X,S) = (X_{1}, S_{1}) \cup (X_{2}, S_{2}),
\]
where $(X_{2}, S_{2})$ is a standard 4-ball containing a standard pair
of disks, and $(X_{1}, S_{1})$ is the closure of the complement. The
two pairs meet along a pair $(S^{3}, U_{2})$, where $U_{2}\subset$ is a
standard 2-component unlink. Let $(X_{2}', S_{2}')$ be obtained from
$(X_{2}, S_{2})$ by the finger move and proper transform. So we have
\[
               (X', S') = (X_{1}, S_{1}) \cup (X_{2}' , S_{2}').
\]

The manifold $(X_{2}' , S_{2}')$ has boundary $(S^{3}, U_{2})$, and we
can form from it a closed pair by attaching a 4-ball and a standard
pair of disks. We write $(Z, \Sigma)$ for the resulting pair:
\begin{equation}\label{eq:Z}
    (Z, \Sigma) = (X_{2}', S_{2}') \cup (B^{4}, D^{2}\amalg D^{2}).
\end{equation}
The manifold $Z$ is a connected sum of two copies of $\bar
  \CP^{2}$, and we write $E_{1}$, $E_{2}$ for the two exceptional
curves. The surface $\Sigma$ is a union of two spheres,
\[
         \Sigma = \Sigma_{1} \amalg \Sigma_{2},
\]
each of which has square $-2$. The two components $\Sigma_{1}$ and
$\Sigma_{2}$ have the same mod 2 homology class, but over the integers
we have (depending on choices made),
\[
\begin{aligned}[]
    [\Sigma_{1}] &= - [E_{1}] - [E_{2}]\\
    [\Sigma_{2}] &= - [E_{1}] + [E_{2}].
\end{aligned}
\]

The proof of the formula for the finger move depends on understanding
the moduli spaces on $(Z,\Sigma)$, for small energy $\kappa$, namely
$\kappa=0$ and $\kappa=1/4$. As in
\eqref{eq:four-w2} above, we will need to understand these moduli
spaces
\[
         M(Z,\Sigma)_{\nu}
\]
for four different values of the Stiefel Whitney class $\nu$, namely
\[
        \nu = 0, \quad \nu =\epsilon_{1}, \quad \nu =\epsilon_{2}, \quad
        \text{and}\quad \nu =\epsilon_{1} + \epsilon_{2},
\]
where $\epsilon_{i}$ is a representative in the class of $E_{i}$. Let
us write
\[
          M(Z, \Sigma)_{*}
\]
for the union of $M(Z,\Sigma)_{\nu}$ over these four values of $\nu$.

Because we are working with $\Isharp$, our interpretation of $M(Z,\Sigma)_{*}$ is that
it parametrizes $\SU(2)$ gauge equivalence classes of anti-self-dual $\SU(2)$
connections on the complement of the four spheres $E_{1}$, $E_{2}$,
$\Sigma_{1}$ and $\Sigma_{2}$, such that the
limiting holonomy around the links of the spheres $\Sigma_{i}$ is order 4, and the
holonomy around the links of the spheres $E_{i}$ are each $1$
\emph{or} $-1$, depending on the value of $\nu$. The metric on $Z$ is
an orbifold metric as usual, with singular set $\Sigma$.

Consider a flat line bundle $\xi$ on
\[
       Z' = Z\setminus ( \Sigma_{1}, \Sigma_{2}, E_{1}, E_{2}).
   \]
Write $(\sigma_{1}, \sigma_{2}, \eta_{1}, \eta_{2})$ for the
holonomy of $\xi$ around the links of these four spheres, and require
that these are
$(\pm 1, \pm 1, \pm 1, \pm 1)$. A push-off of $E_{1}$ meets $\Sigma_{1}$,
$\Sigma_{2}$ and $E_{1}$ once each. Similarly with $E_{2}$.
So we have relations
\[
        \eta_{1}= \eta_{2} =  \sigma_{1} + \sigma_{2}.
    \]
So there are four possibilities for $\xi$,
including the trivial bundle, and their possible holonomies are:
\[
    \begin{aligned}
        (\sigma_{1}, \sigma_{2}, \eta_{1}, \eta_{2}) & =
        (1,1,1,1), \quad\text{or} \\
         & =
         (1,-1,-1,-1), \quad\text{or} \\
          & =
          (-1,1,-1,-1), \quad\text{or} \\
           & =
        (-1,-1,1,1). \\
    \end{aligned}
\]
They form the group isomorphic to $V_{4}$. 

The
flat line bundles act $\xi$ act on $M(Z,\Sigma)_{*}$ by tensor
product. So we have an action of $V_{4}$ on this moduli space.
Tensoring by $\xi$
either leaves $\nu$ unchanged (if $\sigma_{1}=\sigma_{2}$), 
or adds $\epsilon_{1} + \epsilon_{2}$. So
the subset
\begin{equation}\label{eq:zero-and-both}
 M(Z,\Sigma)_{0} \cup  M(Z,\Sigma)_{\epsilon_{1} + \epsilon_{2}}
\end{equation}
is closed under the Klein 4-group action, as is the complementary
subset,
\begin{equation}\label{eq:one-epsilon}
        M(Z,\Sigma)_{\epsilon_{1}} \cup  M(Z,\Sigma)_{\epsilon_{2}}.
    \end{equation}

    The quotient
    \[
            M(Z,\Sigma)_{*} / V_{4}
     \]   
    parametrizes $\SO(3)$ anti-self-dual connections $[B]$ on the complement
    of the four spheres with the property that the holonomy around the
    links of the $\Sigma_{i}$ has order $2$ and the holonomy around the
    links of the $E_{i}$ is $1$ in $\SO(3)$. This is the same as the
    space of bifold $\SO(3)$ connections (without marking) on $(Z,\Sigma)$:
    \[
            M(Z,\Sigma)_{*} / V_{4} = M_{\SO(3)}(Z,\Sigma).
    \]

The next lemma (and the notation ``twisted reducibles'') is from \cite{KM-structure}.

    \begin{lemma}\label{lemma:klein}
        The action of the Klein 4-group is free except at ``twisted
        reducibles''. That is, the $\SU(2)$
        connection $A$ is gauge-equivalent to
        $A\otimes \xi$ if and only if the holonomy of\/
        $[\mathrm{ad}(A)]$ is contained in
        $O(2)\subset\SO(3)$ and the associated
        real line bundle to the $O(2)$ connection is isomorphic to
        $\xi$. \qed
    \end{lemma}
    
The situation described in the lemma above can happen only if $\xi$ either
trivial or has
\[
 (\sigma_{1}, \sigma_{2}, \eta_{1}, \eta_{2})  =
        (-1,-1,1,1).
    \]

We are now ready to describe the small-action moduli spaces, beginning
with a description of the quotients $M_{\SO(3)}(Z,\Sigma)_{*}$.

    \begin{lemma}
        The space of bifold connections $M_{\SO(3)}(Z,\Sigma)$ with
        $\kappa=0$ consists of a single point, with $\Z/2$ monodromy
        and $O(2)$ stabilizer.

        For generic metrics, the space of bifold connections  $M_{\SO(3)}(Z,\Sigma)$ with
        $\kappa=1/4$ consists of a single arc and possibly some
        additional circles. Except for the endpoints of the arc, these
        bifold connections with $\kappa=1/4$ are irreducible (i.e.~have trivial
        stabilizer in $\SO(3)$). The endpoints of the arc have $\SO(2)$
        holonomy, and therefore $\SO(2)$ stabilizer.
    \end{lemma}

   \begin{proof}
        For $\kappa=0$, we are looking at flat orbifold bundles, or
        $\SO(3)$ representations of the orbifold fundamental
        group. The fundamental group of $Z\setminus \Sigma$ is $\Z/2$,
        because this space is $(0,1)\times \RP^{3}$. The two links of
        $S$ are non-zero elements. So there is a unique bifold
        connection.

        For $\kappa=1/4$, consider the branched double cover of
        $\tilde Z \to Z$
        along $\Sigma$. Let $\tilde \Sigma$ the inverse image of $\Sigma$. This
        consists of two spheres $\tilde \Sigma_{i}$, each of
        self-intersection $-1$. The manifold $\tilde Z$ itself is
        diffeomorphic to a connected sum of two copies of
        $\bar\CP^{2}$. 
        On $\tilde Z$ we seek $\SO(3)$
        connections with action $\tilde\kappa=1/2$, which requires
        $w_{2}^{2} = 2$ mod $4$. So $w_{2}=(1,1)$ in the standard
        basis. This is the sort of 1-dimensional moduli space
        considered in \cite{Fintushel-Stern}, from which we learn that 
        the moduli space on $\tilde Z$ with $w_{2}=(1,1)$ and
        $\kappa=1/2$ is 1-dimensional and compact. Its endpoints
        correspond to pairs of integer classes $\pm\lambda$ where
        $\lambda^{2}=-2$ and $\lambda=w_{2}$ mod $2$.
        The only possibilities are $\pm (1,1)$ and $\pm (1,-1)$. So
        the moduli space has two endpoints. The endpoints correspond
        to reducible connections.

        Returning to $Z$, the covering transformation preserves the
        classes $\lambda=(1,1)$ and $(1,-1)$, so the corresponding
        $\SO(2)$ connections on $\tilde Z$ descend to $\SO(2)$
        connections on $Z$. The moduli space on $Z$ is 1-dimensional,
        so must include an arc joining these two points. That is, the
        arc which is contained in the moduli space of $\tilde Z$
        consists of invariant connections which descend to $Z$.
    \end{proof}

   Lemma~\ref{lemma:klein}, together with the description of the
   $\SO(3)$ moduli space in last lemma above, gives
    us a description of the low-dimensional parts of $M(Z,\Sigma)_{*}$:

    \begin{proposition}\label{prop:Z-Sigma-moduli}
        The $\kappa=0$ part of $M(Z,\Sigma)_{*}$ consists of two
        points, each of which has monodromy group the cyclic group
        $\langle \mathbf{i} \rangle \subset \SU(2)$ of order $4$. 
        Under the Klein 4-group action (tensoring by flat line
        bundles), these are each fixed up to gauge equivalence
        by the action of tensoring by by the line bundle 
          $\xi[-1,-1,1,1]$ (in the obvious
        notation from above). The two  connections are interchanged by
        tensoring with $\xi[1,-1,-1,-1]$. These two points belong to
        the moduli spaces $M(Z,\Sigma)_{\epsilon_{1}}$ and $M(Z,
        \Sigma)_{\epsilon_{2}}$.

        The $\kappa=1/4$ part of $M(Z,\Sigma)_{*}$ consists of four
        arcs, together perhaps with some circles. The Klein 4-group
        acts transitively on the four arc-components of the moduli
        space. Two of the arcs belong to $M(Z,\Sigma)_{0}$ and two
        belong to $M(Z,\Sigma)_{\epsilon_{1}+\epsilon_{2}}$. 
    \end{proposition}

    \begin{proof}
        From the previous lemma, the $\kappa=0$ part of the moduli
        space $M(Z,\Sigma)_{*}$ consists of a single orbit of
        $V_{4}$. From Lemma~\ref{lemma:klein} we also learn that the
        stabilizer of the orbit is the two-element subgroup consisting
        of the trivial line bundle and the line bundle
        $\xi[-1,-1,1,1]$. Since the fundamental group of the
        complement of $\Sigma$ is $\Z/2$, there is no flat $\SU(2)$
        bundle on $Z\setminus \Sigma$ whose holonomy on the links of
        $\Sigma$ is conjugate to the element $\mathbf{i}$ of order
        $4$. So the flat $\SU(2)$ connection exists only on
        $Z\setminus (\Sigma \cup E_{1} \cup E_{2})$ and must have
        holonomy $-1$ on the link of exactly one $E_{i}$. These flat
        connections therefore belong to $M(Z,\Sigma)_{\epsilon_{1}}$
        and $M(Z,\Sigma)_{\epsilon_{2}}$.

        We now turn to the $\kappa=1/4$ part of the moduli space. The
        previous lemmas again tell us that the Klein 4-group acts
        freely and the quotient is a 1-manifold containing a single
        arc. Therefore $M(Z,\Sigma)_{*}$ contains 4 arcs. We are left
        to determine which of the four parts of $M(Z,\Sigma)_{\nu}$
        ($\nu=*$) these belong to. An instanton
        $[A] \in M(Z,\Sigma)_{*}$ belonging to one of these arcs pulls
        back to an $\SU(2)$ instanton $[\tilde A]$ on
       \[\tilde Z \setminus (\tilde \Sigma \cup \tilde E_{1} \cup
       \tilde E_{2})\] with limiting holonomy $-1$ on the links of
       $\tilde \Sigma$. The limiting holonomy on the links of the
       sphere $\tilde E_{i}$ will be $(-1)^{\delta_{i}}$, where
       $\delta_{i}=1$ or $0$ according to whether
       $\epsilon_{i}$ appears in $\nu$. Because
       $[\tilde E_{i}]=[\tilde \Sigma_{1}]+[\tilde \Sigma_{2}]$ in mod 2 homology, we then
       obtain
       \[
             w_{2}( \mathrm{ad}(\tilde A) ) = (1 +
             \delta_{1}+\delta_{2})([\tilde \Sigma_{1}] + [\tilde \Sigma_{2}])
        \]
       However, the previous lemma tells us
       that the Stiefel-Whitney class of $[\mathrm{ad}(\tilde A)]$
       is dual to $[\tilde \Sigma_{1}]+[\tilde \Sigma_{2}]$. Therefore
       the possibilities are only $(\delta_{1}, \delta_{2})=(0,0)$ or
       $(1,1)$. The four arcs therefore belong to the components 
      $M(Z,\Sigma)_{0}$ and
      $M(Z,\Sigma)_{\epsilon_{1}+\epsilon_{2}}$. Two lie in each,
      because of the symmetry that arises from the $V_{4}$ action.
    \end{proof}

    \begin{corollary}\label{cor:opposite}
        If $A$ and $A'$ are the two (abelian) connections which comprise
        the zero-dimensional part of $M(Z,\Sigma)_{*}$, and if $m_{1}$ and
        $m_{2}$ are links of the two components of\/ $\Sigma$, oriented so
        that $A$ has monodromy $\mathbf{i}$ around both links, then the
        monodromy of $A'$ around $m_{1}$ and $m_{2}$ are $\mathbf{i}$ and $-\mathbf{i}$,
        up to overall conjugacy. \qed
    \end{corollary}

    Let us return now to the pair $(X'_{2}, \Sigma'_{2})$, which we
    equip with a cylindrical end
    $\R^{+}\times (S^{3}, U_{2})$. With notation adapted from the
    discussion of $(Z,\Sigma)$, we examine the moduli space
         \[
                 M(X'_{2}, S'_{2})_{*}
          \]
   on the cylindrical end moduli spaces, with Stiefel-Whitney class
   $\nu=*$ running over the same four values. The $\SU(2)$ representation
   variety of $(S^{3}, U_{2})$ is an interval, which we denote by
   $[-1,1]$, so we have a map
  \[
            r :  M(X'_{2}, S'_{2})_{*} \to [-1,1]
  \]
    
   From
   Proposition~\ref{prop:Z-Sigma-moduli} and a stretching argument we
   learn that the $\kappa=0$ part $M(X'_{2}, S'_{2})_{*}$ consists
   of two points which are mapped by $r$ to endpoints of the interval $[-1,1]$. From
   Corollary~\ref{cor:opposite} we learn that the two points map to
   opposite ends of the moduli space.

   Similarly we learn that the $\kappa=1/4$ part of $M(X'_{2},
   S'_{2})_{*}$ contains four arcs, and that these are each mapped
   to $[-1,1]$ in such a way that the two endpoints of each arc map to
   opposite ends of $[-1,1]$. 

   Having described  these moduli spaces on the cylindrical-end
   manifold, we now describe how these give rise to the formula in
   Proposition~\ref{prop:twist-and-finger-formulae} for the case of
   the finger move. We can break the formula up into:
   \begin{enumerate}
   \item terms coming from the classes $\nu=0$ and
       $\nu=(\epsilon_{1} + \epsilon_{2})$ on the one hand; and
   \item terms coming from the classes $\nu=\epsilon_{1}$ and
       $\nu=\epsilon_{2}$,
   \end{enumerate}
(\cf~equations \eqref{eq:zero-and-both} and \eqref{eq:one-epsilon} above).
The first case is that of the four arcs that comprise the $\kappa=1/4$
   moduli space. Here the discussion closely mirrors the argument for the
   finger move in \cite{Obstruction} and \cite{KM-singular}. Each of
   the four arcs will contribute term to the formula having the shape
   \[
                    T_{0}^{x} T_{0}^{y} \Isharp(X, S; \Gamma)_{\omega}
   \]
 where $x$ and $y$ are curvature integrals for $\SO(2)$ connections on
 the components $\Sigma'_{1}$, $\Sigma'_{2}$ of the singular set in
 the cylindrical-end bifold $(X_{2}',S'_{2})$. The action of the group
 $V_{4}$ is by tensoring with real line bundles, the effect of which
 is to change the signs of $x$ and $y$. So the formula for the four
 arcs together has the form
\[
            (  T_{0}^{x}T_{0}^{y} +  T_{0}^{-x}T_{0}^{y} +
            T_{0}^{x}T_{0}^{-y} +  T_{0}^{-x}T_{0}^{-y})\, \Isharp(X, S; \Gamma)_{\omega}.
\] 
By symmetry, we have $x=y$ up to sign. So the formula simplifies to
\[
           (  T_{0}^{2x} +  T_{0}^{-2x})\, \Isharp(X, S; \Gamma)_{\omega}.
\]
A special case of the finger move is a pair of twist moves, one
positive and one negative. So by comparing this formula to the case of
the twist moves, we see that $x=\pm 1$. So the contribution of the
$\kappa=1/4$ moduli spaces to the formula for $\Isharp(X,S^{*};\Gamma)_{\omega}$ is
\begin{equation}\label{eq:arc-terms}
            (  T_{0}^{2} +  T_{0}^{-2})\, \Isharp(X, S; \Gamma)_{\omega}.
\end{equation}

Turning finally to the contributions from the classes
$\nu=\epsilon_{1}$ and $\nu=\epsilon_{2}$, we have seen that the
$\kappa=0$ moduli spaces
\[
      M(X_{2}, S_{2}')_{\epsilon_{1}} \quad\text{and}\quad 
                M(X_{2}, S_{2}')_{\epsilon_{2}} 
\]
each consist of a single point, and these map to the two endpoints of
$[-1,1]$. We can compare this to the moduli space $M(X_{2}, S_{2})$
with $\kappa=0$ (where $(X_{2}, S_{2})$ is now a ball with two disks,
equipped with a cylindrical end). For the latter moduli space, the map
\[
           r: M(X_{2}, S_{2})_{\kappa=0} \to [-1,1]
\]
is a homeomorphism. Let us pick points $p$ and $q$ on the two
disks and orientation $o_{p}$ and $o_{q}$ nearby. 
The flat connections on the orbifold $(X_{2}, S_{2})$ are determined
by the holonomies around oriented meridians at this point, as in
section~\ref{subsec:Lambda-class}, or equivalently by unit vectors
$i_{p}$ and $i_{q}$ in the $\R^{3}$ fibers $\mathbb{E}_{p}$ and
$\mathbb{E}_{q}$. Under the homeomorphism $r$, the endpoints of the
interval correspond to flat connections with $i_{p}=i_{q}$ and
$i_{p}=-i_{q}$ (when we identify $\mathbb{E}_{p}$ with
$\mathbb{E}_{q}$ via paths to the basepoint). If we work with based
the based moduli space of flat connections on $(X_{2}, S_{2})$, then
we instead obtain a map
\[
           r: \tilde M(X_{2},S_{2})_{\kappa=0} \to [-1,1]
\]
where the domain is now $S^{2}\times S^{2}$ and the preimage of the
endpoints is the union of the diagonal and anti-diagonal. This is
precisely the intersection of $  \tilde M(X_{2},S_{2})_{\kappa=0} $
with the standard representatives of the 2-dimensional cohomology
classes $\lambda_{pq}$ and $\lambda'_{pq}$ from
\eqref{eq:lambda-class-pq}. If we recall the relation $\lambda_{pq} +
\lambda'_{pq}= w_{2}(\mathbb{E}_{p})$ from
Lemma~\ref{lemma:lambda-relations-cohomology}, then we learn that the
preimage of the two endpoint comprise a standard representative for
the Poincar\'e dual of $w_{2}$. From Lemma~\ref{lemma:w-relations} and
a stretching argument, it then follows that 
for the original closed pair $(X,S)$ and the pair
$(X,S^{*})$ obtained by the finger move, the contribution to
$\Isharp(X,S^{*} ; \Gamma)_{\omega}$ coming from these moduli spaces
is $P \, \Isharp(X, S; \Gamma)_{\omega}$. This formula and the terms
\eqref{eq:arc-terms} together give the formula in
Proposition~\ref{prop:twist-and-finger-formulae} for the finger move:
\begin{equation}\label{eq:finger-formula}
        \Isharp(X, S^{*}; \Gamma)_{\omega} = ( P +  T_{0}^{2} +
        T_{0}^{-2}) \,\Isharp(X, S ; \Gamma)_{\omega},
\end{equation}
or more succinctly
\begin{equation}\label{eq:finger-formula-L}
        \Isharp(X, S^{*}; \Gamma)_{\omega} =L  \,\Isharp(X, S ; \Gamma)_{\omega}.
\end{equation}

\begin{remark}
    As discussed in section~\ref{subsec:double-points} above, we can
    choose to
    change our definition for the blow-ups and use the formula
    \eqref{eq:blow-up-rule-xi}. A little extra book-keeping is then
    required, but the final result needs only slight modification. 
    For the resulting functor $\Isharp_{\xi}$, the statement  of 
   Proposition~\ref{prop:twist-and-finger-formulae} is unchanged
   except for the formula for the factor $L$. We record this as a
   proposition.

\begin{proposition}
   \label{prop:twist-and-finger-formulae-xi}
    As in Proposition~\ref{prop:twist-and-finger-formulae}, 
    let $S^{*}$ be obtained from $S$ by either a positive
    twist move, or a finger move. Let the modified functor
    $\Isharp_{\xi}$ be defined using the blow-up rule
    \eqref{eq:blow-up-rule}. Then we have,
    \[
              \Isharp_{\xi}(X, S^{*} ; \Gamma_{\s})_{\omega}  = \s(\V_{\xi})\,
          \Isharp_{\xi}(X, S ; \Gamma_{\s})_{\omega} 
    \]
    where
     \[ \V_{\xi}  =  \xi P + T_{0}^{2} +
        T_{0}^{-2}.
   \] 
   For the negative twist move, the map $\Isharp_{\xi}$ is again
   unchanged:
   \[
              \Isharp(X, S^{*} ; \Gamma_{\s})_{\omega}  =
          \Isharp(X, S ; \Gamma_{\s})_{\omega} .
   \]
 \end{proposition}
\end{remark}

\subsection{Regular homotopies}

Recall that if $f_{0}$ and $f_{1}$ are two smooth embeddings of a closed
surface $S$ in a 4-manifold $X$, and if $f_{0} \simeq f_{1}$ as maps,
then one can find a homotopy which is a composite of steps, each of
which is one of:
\begin{itemize}
\item the introduction of a transverse double-point by a twist move;
\item the introduction of two transverse double-points by a finger
    move;
\item the inverse to one of the above;
\item an ambient isotopy.
\end{itemize}
The same applies to surfaces $S$ which arise as cobordisms between
knots or links, when the homotopy is relative to the boundary. As in
\cite{Obstruction,KM-singular,KM-s-invariant}, this observation can
be combined with the formulae in
Proposition~\ref{prop:twist-and-finger-formulae}, to obtain the
following result (among others).

\begin{proposition}\label{prop:homotopy}
    Let $S \subset \R^{4}$ be a closed embedded surface, not
    necessarily connected. Regard $S$ as a cobordism from the empty
    link in $\R^{3}$ to itself, optionally equipped with dots
    $q_{1}$,\dots, $q_d$. Then the resulting
    map
    \[
             \Isharp((S; q_{1}, \dots, q_{d}) ; \Gamma_{\s}) : \cS
             \to \cS
     \]
   depends only on the topology of the components of $S$, the number
   of dots on each, and the local orientations. \qed
\end{proposition}

\section{The unknot and unlinks}

The instanton homology
$\Isharp(K; \Gamma_{\s})$ is a free $\cS$-module of rank $2$ when $K$ is the
unknot, and it is a free module of rank $2^{n}$ for the
$n$-component unlink. Although establishing these statements is not
hard, we will need a little more for application in our spectral sequence in the
following section: we need to make these isomorphisms canonical, to
the extent that is possible. For this task, our exposition will follow
\cite{KM-unknot} to begin with. However, there is a little more
subtlety now, even in the case of the unknot. This stems in
part from the fact that $\Isharp(K ;\Gamma)$ is only $\Z/2$ graded
(there is no $\Z/4$ grading as there was in \cite{KM-unknot}), and
the two generators for the unknot are in the same grading mod $2$, so
we cannot use the grading decomposition to pick out canonical generators.

\subsection{Spheres with dots}

Let $S \subset \R^{4}$ be an embedded sphere. Choose one orientation,
and let $q_{1}$, \dots, $q_{d}$ be dots on $S$ whose orientation
agrees with the chosen orientation of $S$. We wish to evaluate
the corresponding map on the homology of the empty link, which we
regard as defining an element
\[
       \Isharp((S; q_{1}, \dots, q_{d}); \Gamma_{\s}) \in \cS
\]
where $\Gamma_{\s} = \Gamma\otimes_{\s} \cS$. By
Proposition~\ref{prop:homotopy}, the evaluation is independent of the embedding.

\begin{lemma}\label{lem:S-eval}
    The evaluation $\epsilon_{d}$ of the sphere with $d$ dots is $0$ for $d=0$, and\/
    $1$ for $d=1$. For $d\ge 2$, it satisfies the recurrence relation
   \[
           \epsilon_{d} = \s(P) \epsilon_{d-1} + \s(Q) \epsilon_{d-2}. 
    \]
\end{lemma}

\begin{proof}
    The formal dimension of the relevant moduli space is positive when
    the 
    Yang-Mills action $\kappa$ is zero, so for $d=0$ the evaluation is
    zero. For $d=1$, we use the fact that the $\kappa=0$ moduli space
    parametrizes flat connections and is a $2$-sphere when $S$ has the
    standard embedding. The cohomology class $\lambda_{q}$ is set up
    so that it evaluates to $1$ on this $2$-sphere. For $d\ge 2$, the
    recurrence relation follows from Lemma~\ref{cor:Lambda-reln}.
\end{proof}

\subsection{The empty knot and the unknot}

As in \cite{KM-unknot}, we write $U_{n}$ for a standard unlink in
$\R^{3}$ with $n$ components, so that $U_{0}$ is the empty link and
$U_{1}$ is the unknot. We take $U_{n}$ to be the union of standard
circles in the $(x,y)$ plane, each of diameter $1/2$, and centered on
the first $n$ integer lattice points along the $x$ axis.  We orient
the circles of $U_{n}$ by a standard choice, say anti-clockwise in the
$(x,y)$ plane.

For the empty link, $\Isharp(U_{0};\Gamma_{\s})$ is free of rank $1$,
and we can canonically choose an identification with $\cS$, or
equivalently a generator 
\[
       \bu_{0} \in \Isharp(U_{0};\Gamma_{\s}).
\]

\begin{lemma}\label{lem:unknot-1}
    For the unknot $U_{1}$, the instanton homology $\Isharp(U_{1}; \Gamma_{\s})$ is free of rank
    $2$. As generators, we can take the image of $\bu_{0}$ under the
    two maps
   \[
            \Isharp(U_{0};\Gamma_{\s}) \to  \Isharp(U_{1};\Gamma_{\s})
   \]
 given by (a) a standard disk $D^{+}$ with boundary $U_{1}$; or (b)
 the disk $D^{+}$ decorated with a dot $q$ whose local orientation
 arises from our choice of orientation for the knot.
 \end{lemma}

 \begin{proof}
     The Chern-Simons functional has a perturbation with just two
     critical points. So the rank is at most $2$; and equality can
     hold only if it is a free module. Let $D^{-}$ be a
     disk providing a cobordism from $U_{1}$ to $U_{0}$, and let $q'$
     be a dot on $D^{-}$. Using Lemma~\ref{lem:S-eval}, we can
     compute the pairings between the cobordisms $D^{+}$, $(D^{+}, q)$
     on the one side, and the cobordisms $D^{-}$, $(D^{-}, q')$ on the
     other. The result is the matrix
     \[
     \begin{pmatrix}
         0 & 1 \\ 1 & P
     \end{pmatrix},
     \]
     whose determinant is $1$. It follows that the rank of the module
     is $2$, the images of $D^{+}$ and $(D^{+}; q)$ are generators.
 \end{proof}

 \begin{definition}\label{def:xy-U1}
     We write $V$ for the rank-2 $\cS$-module $\Isharp(U_{1};
     \Gamma_{\s})$. Define \[\bx_{+}, \; \bx_{-} \in \Isharp(U_{1}; \Gamma_{\s})\]
     to be the images of $\bu_{0}$ under the maps arising from the
     cobordisms $D^{+}$ and $(D^{+}; q)$. They form a basis for this
     free module, by the lemma. In the dual module, we
     define
     \[
              \by_{+} , \; \by_{-}:  V \to \cS
      \]
     using respectively the cobordisms $D^{-}$ and $(D^{-}, q)$.
    \qed
 \end{definition}

The proof of the previous lemma gives the pairings between $\bx_{\pm}$
and $\by_{\pm}$, and from the knowledge of those pairings we obtain:

\begin{lemma}
    The dual basis to the basis $(\bx_{+}$, $\bx_{-})$ for the free
    module $V=\Isharp(U_{1};\Gamma_{\s})$ is the basis $(\by_{-} + \s(P) \,\by_{+},  \by_{+})$.\qed
\end{lemma}

\subsection{The homology of the unlink}

Having identified $\Isharp(U_{1}; \Gamma_{\s})$ as the free module
$V=\langle \bx_{+}, \bx_{-}\rangle$, we can examine the $n$-component
unlink $U_{n}$ using the strategies from \cite{KM-unknot}.

\begin{lemma}[Corollary~8.5 of \cite{KM-unknot}]
\label{lemma:unlink-tensor-n}
We have an isomorphism of\/ $\cS$-modules,
    \[
        \Phi_{n}: V^{\otimes n} \to \Isharp(U_{n};\Gamma_{\s})       ,
     \]
      for all $n$,
     with the following  properties. First, 
     if\/ $D^{+}_{n}$  denotes the cobordism from $U_{0}$ to $U_{n}$
     obtained from standard disks as in the previous lemma, then
     \[
                  \Isharp(D^{+}_{n};\Gamma_{\s})(\bu_{0}) = 
                           \Phi_{n}( \bx_{+}\otimes \dots \otimes \bx_{+}).
    \]
     Second, the isomorphism is natural for split cobordisms, perhaps
     with dots,  from
     $U_{n}$ to itself. Here, a ``split'' cobordism means a cobordism
     from $U_{n}$ to $U_{n}$ in $[0,1]\times \R^{3}$ which is the
     disjoint union of $n$ cobordisms from $U_{1}$ to $U_{1}$, each
     contained in a standard ball\/ $[0,1]\times B^{3}$. 
\end{lemma}

\begin{proof}
    This is essentially the same as the version in
    \cite{KM-unknot}. Note that the trivial cobordism from $U_{1}$ to
    $U_{1}$, equipped with a dot $q$ and an appropriate local
    orientation, gives the map $\Lambda_{q}:V\to V$ which maps $\bx_{+}$ to
    $\bx_{-}$. 
\end{proof}

The next lemma and its corollary are also drawn directly from
\cite{KM-unknot}, and establish that the isomorphism of the previous
lemma is canonical, once the unlink has been oriented.

\begin{lemma}
    Let $S$ be an oriented concordance from the standard unlink $U_{n}$ to
    itself, consisting of $n$ oriented annuli in $[0,1]\times
    \R^{3}$. 
    Let $\tau$ be the permutation of $\{1,\dots,n\}$
    corresponding to the permutation of the components of $U_{n}$
    arising from $S$. Then the standard isomorphism $\Phi_{n}$ of
    Lemma~\ref{lemma:unlink-tensor-n} intertwines the map
     \[
              \Isharp(S;\Gamma_{\s}) : \Isharp(U_{n};\Gamma_{\s}) \to \Isharp(U_{n};\Gamma_{\s})
     \]
     with the permutation map
     \[
                 \tau_{*} : V\otimes \dots\otimes V \to  V\otimes
                 \dots \otimes V.
      \]
      In particular, if the permutation $\tau$ is the identity, then
      $\Isharp(S;\Gamma_{\s})$ is the identity.
\end{lemma}

\begin{proof}
    The proof leverages Proposition~\ref{prop:homotopy}, and is the
    same as the proof in \cite{KM-unknot}, with the dot
    operator $\Lambda_{q}$ replacing the operator $\sigma$ (equation
    (56) in \cite{KM-unknot}).
\end{proof}

\begin{corollary}\label{cor:unlink-canonical}
    Let $\cU_{n}$ be any oriented link in the link-type of\/ $U_{n}$, and
    let its components be enumerated. Then there is a canonical
    isomorphism
    \[
             \Psi_{n} :  V\otimes \dots \otimes V \to \Isharp(\cU_{n};\Gamma_{\s})
    \]
    which can be described as $ \Isharp(S;\Gamma_{\s}) \circ \Phi_{n}$, where
    $\Phi_{n}$ is the standard isomorphism of
     Lemma~\ref{lemma:unlink-tensor-n} and $S$ is any cobordism from
    $U_{n}$ to $\cU_{n}$ arising from an isotopy from $U_{n}$ to
    $\cU_{n}$, respecting the orientations and the enumeration of
    the components. 
    
    If the enumeration of the components of $\cU_{n}$ is changed by a
    permutation $\tau$, then the isomorphism $\Psi_{n}$ is changed
    simply by composition with the corresponding permutation of the
    factors in the tensor product. \qed
\end{corollary}

The corollary tells us that the homology of the unlink is canonically
isomorphic to the tensor product once an orientation of the components
has been chosen. The last thing we need to do here is determine the
dependencs of the isomorphism on the choice of orientation.

\begin{proposition}
    Let $S: U_{1} \to U_{1}$ be a cobordism arising from an isotopy of
    the standard unknot to itself which reverses the orientation. Then
    the resulting map $\iota=\Isharp(S; \Gamma_{\s}) : V \to V$ is given by
\[
\begin{aligned}
   \iota  : \bx_{+} &\mapsto \bx_{+} \\
    \iota  : \bx_{-} &\mapsto \s(P) \,\bx_{+} + \bx_{-} .
\end{aligned}
\]
\end{proposition}

\begin{proof}
    Recall that $\bx_{-}=\Lambda_{q}\bx_{+}$. The cobordism $S$ intertwines the operator $\Lambda_{q}$ with
    $\Lambda'_{q}$. The formula for $\iota$ therefore follows from
    the relation $\Lambda_{q} + \Lambda_{q}'=\s(P)$ in
    Corollary~\ref{cor:norm-Lambda-reln}. 
\end{proof}

\subsection{Pants and copants}
\label{subsec:pants-copants}

Recall the standard cobordism called ``pants'', from the two-component
unlink to the one-component unknot:
\[
           \pants: U_{2} \to U_{1}.
\]
Its mirror image is ``copants'',
\[
         \copants: U_{1} \to U_{2}.
\]
If we identify $\Isharp(U_{1};\Gamma_{\s})$ and
$\Isharp(U_{2};\Gamma_{\s})$ with $V$ and $V\otimes V$ by the canonical
isomorphisms of Corollary~\ref{cor:unlink-canonical}, then pants and
copants give rise to maps
    \begin{equation}\label{eq:Isharp-pants}
    \begin{gathered}
        \Isharp(\pants; \Gamma_{\s}) : V\otimes V \to V \\
 \Isharp(\copants; \Gamma_{\s}) : V \to V\otimes V .
    \end{gathered}
    \end{equation}

\begin{proposition}\label{prop:pants-copants}
    Under the above identification, the maps arising from the pants 
    cobordism $\pants$ is given by:
\begin{equation}\label{eq:multn}
\begin{aligned}
    \bx_{+} \otimes \bx_{+} &\mapsto \bx_{+} \\
    \bx_{+} \otimes \bx_{-} &\mapsto \bx_{-} \\
    \bx_{-} \otimes \bx_{+} &\mapsto \bx_{-} \\
    \bx_{-} \otimes \bx_{-} &\mapsto \s(P) \bx_{-} + \s(Q) \bx_{+} .
\end{aligned}
\end{equation}
The map arising from the copants cobordism $\copants$ is:
\begin{equation}\label{eq:comultn}
\begin{aligned}
    \bx_{+}  &\mapsto \bx_{+}\otimes \bx_{-} + \bx_{-}\otimes \bx_{+} +
   \s( P)\bx_{+}\otimes \bx_{+}\\
   \bx_{-}  &\mapsto \bx_{-}\otimes \bx_{-} + \s(Q) \bx_{+} \otimes \bx_{+} .
\end{aligned}
    \end{equation}
\end{proposition}

\begin{proof}
   Using the standard basis
    elements $\bx_{\pm}$ and dual basis $\by_{\pm}$, we can reduce
    this to the evaluation of a 2-sphere with dots, for which we have
    the formulae in Lemma~\ref{lem:S-eval}. 
\end{proof}

\subsection{The reduced homology of the unlink}

Let $\s: \cR\to\cS$ be a base change with $\s(T_{0})=\s(T_{1})$, so
that the reduced instanton homology $\Inat(K ; \Gamma_{\s})$ is
defined for a link $K$ with base-point.  As with other
``reduced'' versions of knot homologies, from the
definitions, $\Inat(U_{1} ; \Gamma_{\s}) \cong \Isharp(U_{0};
\Gamma_{\s}) \cong \cS$, and
\begin{equation}\label{eq:reduced-unlink}
    \begin{aligned}
        \Inat(U_{n}; \Gamma_{\s}) &\cong \Inat(U_{1} ; \Gamma_{\s})
        \otimes_{\cS} \Isharp(U_{n-1} ; \Gamma_{\s} )\\
        &\cong \Isharp(U_{n-1} ; \Gamma_{\s}).
    \end{aligned}
\end{equation}
In particular, $\Inat(U_{n}; \Gamma_{\s})$ is a free module of rank $2^{n-1}$.
We would like to compute the maps on  $\Inat(U_{n}; \Gamma_{\s})$
given by the pants and copants cobordisms, particularly when one the
incoming components of the cobordisms carries the base-point.

As a first step, we consider again the operator $\Lambda_{q}$, for
$q\in K$, now as an operator on $\Inat(K;\Gamma_{\s})$. We define
this as before, as in Definition~\ref{def:Lambda3},
\[
             \Lambda_{q} = \Lambda_{p_{1}q } + \Lambda_{p_{2}q } + \Lambda_{p_{3}q }
\]
where the three $p_{i}$ are dots chosen near the vertex, so that
$p_{2}$ and $p_{3}$ lie on the two edges that form the bigon in
$K^{\natural}$. 

If it happens that $q$ lies on the component of $K$ where the bigon
is attached,
we can take
$q=p_{1}$, in which case the first term $\Lambda_{p_{1}p_{1}}$ is
$P$. In this setting, to compute the operator, it is sufficient to
examine the case that $K$ is the unknot, by excision;
so $K^{\natural}$ can be taken to be the theta
graph.
Each of the three terms is then an operator $\cS\to \cS$, so
altogether $\Lambda_{q}$ is a multiplication operator,
\[
            \Lambda_{q} = A : \cS \to \cS.
\]
To compute $A$, we seeking to compute
\begin{equation}\label{eq:A-first}
           \Lambda_{q} = P + \Lambda_{p_{2} p_{1} } + \Lambda_{p_{3}p_{1} }.
\end{equation}
The operator $\Lambda_{p_{2}p_{1}}$ for the theta graph was computed
in \eqref{eq:local-lambdas}, up to a choice of two possibilities,
differing by $P$. The same ambiguity is present twice in
\eqref{eq:A-first}, for $ \Lambda_{p_{2} p_{1} }$ and for
$\Lambda_{p_{3} p_{1} }$, 
and it is resolved the same way in both terms. So the ambiguity
cancels, and we are left with a unique formula,
\[
     \Lambda_{q} = P +  T_{3}(T_{1}T_{2} + 
   T_{1}^{-1}T_{2}^{-1}) +  T_{2}(T_{1}T_{3} + T_{1}^{-1}T_{3}^{-1}) 
\]
which simplifies to
\[
              A = T_{1} (T_{2}T_{3} + T_{2}^{-1}T_{3}^{-1}).
\]
(We have omitted the base change $\s$ in our notation.) To summarize
this calculation in the case of the unknot, we have the following.

\begin{proposition}\label{prop:reduced-calc}
    For the unknot $U_{1}$, the reduced homology $\Inat(U_{1};
    \Gamma_{\s})$ is a free $\cS$-module of rank $1$, on which the
    operator $\Lambda_{q}$ ($q\in U_{1}$) acts as multiplication by
    $A$, where $A$ is the element above.
\end{proposition}

What lies behind the algebra here is the following observation. The
operator $\Lambda$ on the un-reduced homology $V = \Isharp(U_{1};
\Gamma_{\s})$ has minimum polynomial
\[
                  x^{2} + \s(P)x + \s(Q).
\]
When the base change has $\s(T_{0})=\s(T_{1})$, the minimum polynomial
factorizes as
\[
               (x + A)(x + A')
\]
where $A$ is as above and $A'=A+P$. This is the same observation as we
made in the introduction to this paper, at \eqref{eq:A-formulae}. 
Let us define $V^{\natural}
\subset V$ as
\[
\begin{aligned}
    V^{\natural} &= \ker( \Lambda+A) \\
    &= \im (\Lambda+A').
\end{aligned}
\]
So $V^{\natural}$ is the rank-1 $\cS$-submodule generated by the
element \[\bm=\bx_{-}
+ A' \bx_{+}.\] Then we have:

\begin{corollary}\label{cor:reduced-calc}
    The reduced homology  $\Inat(U_{1};
    \Gamma_{\s})$ is isomorphic as a module for $\cS[\Lambda]$ to the 
    submodule $V^{\natural} \subset  V$ generated by\/ $\bm$ above.
\end{corollary}

We can consider next the pants and copants cobordisms in the reduced
context. Let $U_{2}$ be the standard 2-component unlink with a
basepoint on the first component. We have, by excision,
\[
           \Inat(U_{2} ;\Gamma_{\s}) = V^{\natural} \otimes_{\cS} V.
\]
The pants and copants cobordisms provide maps
    \[
    \begin{gathered}
        \Inat(\pants; \Gamma_{\s}) : V^{\natural}\otimes V \to V^{\natural} \\
 \Inat(\copants; \Gamma_{\s}) : V^{\natural} \to V^{\natural} \otimes V .
    \end{gathered}
    \]
It is straightforward to verify that these coincide with the
restriction of the un-reduced versions \eqref{eq:Isharp-pants} to the
$\cS$-submodule $V^{\natural}$ generated by $\bm$. We can write these
maps out, in terms of the basis $\{\, \bx_{+}, \bm\,\}$ for the rank-2
$\cS$-module $V$:

\begin{proposition}\label{prop:pants-copants-reduced}
    When $\s(T_{0})=\s(T_{1})$ so that the reduced theory is defined, 
    the map  $V^{\natural}\otimes V \to V^{\natural}$ arising from the pants 
    cobordism 
\begin{equation}\label{eq:multn-red}
\begin{aligned}
    \bm \otimes \bx_{+} &\mapsto \bm \\
    \bm \otimes \bm &\mapsto P \, \bm .
\end{aligned}
\end{equation}
  The map $V^{\natural}\to V^{\natural}\otimes V$ arising from copants
  is
\begin{equation}\label{eq:comultn-red}
\begin{aligned}
    \bm  &\mapsto \bm \otimes \bm.
\end{aligned}
\end{equation}
\end{proposition}

Returning again to $V$ in the unreduced case, we have a description of
it as an algebra over $\cR$ with a single generator $\n = \bx_{-}$ in
the form
\[
            V = \cR[\n] / (\n^{2} + P \/\n + Q).
\]    
As in the introduction, after a base change to a ring $\cS$ where $T_{0}=T_{1}$,
the characteristic polynomial $(x^{2} + Px +Q)$ factorizes as
$(x+A)(x+A')$ and over $\cS$ we have a presentation
\[
              V = \cS[\bm] / (\bm(\bm + P)).
\]    
The full co-multiplication of the Frobenius algebra $V$, arising from
the copants cobordism, in this
presentation is
\[
    \begin{aligned}
       \Delta: \mathbf{1} &\mapsto \bm\otimes \mathbf{1} +  \mathbf{1}\otimes
        \bm + P\/\mathbf{1}\otimes  \mathbf{1} \\
        \Delta: \bm &\mapsto \bm\otimes\bm.
    \end{aligned}
 \]
This is the Frobenius algebgra that gives rise to the graded Bar-Natan
variant of Khovanov homology, tensored by $\cS$.
 
\section{The spectral sequence}

\subsection{Families of metrics.}
Relevant to the construction of our spectral sequence are also the
maps that arise from a cobordism equipped with a family of
metrics. The material of \cite[section 3.9]{KM-unknot} again adapts to
local coefficients without change. We equip $(X,S)$ with
cylindrical ends and a family
of Riemannian metrics $G$ which vary only in a compact region. The
parameter space $G$ should be a compact manifold with boundary. After
choosing perturbations, the moduli spaces over $G$ define
homomorphisms of $\cS$-modules,
\[
        m_{G} : C^{\sharp}(Y_{0}, K_{0}; \Gamma_{\s}) \to   C^{\sharp}(Y_{1}, K_{1}; \Gamma_{\s}),
\]
where the complexes $ C^{\sharp}(Y_{i}, K_{i}; \Gamma_{\s})$ are those that
compute $\Isharp$.
The map $m_{G}$ is a chain map if $G$ has no
boundary. Otherwise, there is an extra term in the chain formula,
\[
  m_{\partial G} +   m_{G}\circ \partial =  \partial \circ m_{G}.
\]
(See  \cite[section 3.9]{KM-unknot} and \cite{KM-deformation}.)

\subsection{Skein exact triangle}

Fix again a $3$-manifold $Y$ with basepoint $y_{0}$ and a theta graph
$\theta\subset B(y_{0})$. We again write $Y^{o} \subset Y$ for the
complement of the neighborhood of $y_{0}$. Consider three webs $K_{2}$, $K_{1}$,
$K_{0}$ in $Y^{o}$ which are all identical outside a ball
$B\subset Y^{o}$ and which differ inside $B$ by the
skein moves as shown in Figure~\ref{fig:Tetrahedra-skein}. There are
standard cobordisms $S_{ij}$ from $K_{i}$ to $K_{j}$, each of which is
the addition of a standard $1$-handle in $[0,1]\times B$. Although the
webs may have vertices, there are  no vertices in the ball $B$, and
the picture coincides with that of \cite[section 6]{KM-unknot}. As in
\cite{KM-unknot,KM-deformation}, the cobordisms $S_{21}$, $S_{10}$ and $S_{02}$ give
rise to the maps in a 3-periodic long exact sequence of $\cS$-modules:
\[
          \cdots \to \Isharp(K_{2}; \Gamma_{\s}) \to \Isharp(K_{1}; \Gamma_{\s}) \to
          \Isharp(K_{0}; \Gamma_{\s}) \to  \Isharp(K_{2}; \Gamma_{\s}) \to \cdots.
\]

\subsection{Cubes of resolutions.} The above skein sequence can be seen
as a consequence of the fact that the chain complex $C^{\sharp}_{2}$
that computes $\Isharp(K_{2}; \Gamma_{\s})$ is quasi-isomorphic to  the
mapping cone of a chain map $C^{\sharp}_{1}\to C^{\sharp}_{0}$. As in
\cite{KM-unknot}, the skein sequence generalizes as follows. Suppose
that $Y^{o}$ contains $N$ disjoint balls $B_{1}$, \dots, $B_{N}$. For
each $v\in \{0,1,2\}^{N}$, let there be given a web
$K_{v}\subset Y^{o}$. Outside the balls, all the $K_{v}$ are the
same. Inside the ball $B_{i}$, the web $K_{v}$ coincides with on of
the models in Figure~\ref{fig:Tetrahedra-skein}, according to the
value of the coordinate $v_{i}$. We write $(C^{\sharp}_{v}, d_{v})$
for the standard chain complex that computes
$\Isharp(K_{v} ; \Gamma_{\s})$. (A choice of metric and perturbation is
involved.)

Among the $K_{v}$, we pick out as distinguished the web
$K_{\mathbf{2}}$, where 
\[
    \mathbf{2} = (2,2,\dots, 2).
\]
We also introduce the ``cube''
\[
          \bC^{\sharp} = \bigoplus_{v\in \{0,1\}^{N}} C^{\sharp}_{v}.
\]
For each $v > u$ in $\{0,1\}^{N}$, there is a standard cobordism
$S_{vu}$ from $K_{v}$ to $K_{u}$, obtained by adding $1$-handles in
each of the balls $B_{i}$ where the coordinates of $v$ and $u$
differ. If there are $n$ such coordinates, then $S_{vu}$ carries a
standard family $G_{vu}$ of metrics, of dimension $n-1$, as described
in \cite{KM-unknot}, which give rise to $\cS$-module homomorphisms
\[
           f_{vu} : C^{\sharp}_{v} \to C^{\sharp}_{u}.
\]
As a special case, we also define $f_{vv}=d_{v}$. We then define
\[
         \bF : \bC^{\sharp} \to \bC^{\sharp}
\]
as 
\[
           \bF = \bigoplus_{v\ge u} f_{vu}.
\]

\begin{theorem}[Theorem 6.8 of \cite{KM-unknot}]\label{thm:big-cube}
    The square of\/ $\bF$ is zero, so $(\bC^{\sharp}, \bF)$ is a complex
    of $\cS$-modules.
    Furthermore,  there
    is a chain map
\[
          (C^{\sharp}_{\mathbf{2}}, d_{\mathbf{2}}) \to (\bC^{\sharp}, \bF)
\]
   inducing an isomorphism in homology. In particular, the homology of
   the ``cube'' complex $ (\bC^{\sharp}, \bF)$ is isomorphic to
   $\Isharp(K_{\mathbf{2}} ; \Gamma_{\s})$.
\end{theorem}

As is standard in Khovanov homology, the cube $\bC^{\sharp}$ has a
filtration (increasing, with our conventions),
\[
            \cF_{n} \bC^{\sharp} = \bigoplus_{
                                    \substack{v\in\{0,1\}^{N} \\ |v|
                                      \le n}} C^{\sharp}_{v}.
\]
There is a corresponding spectral sequence, just as in \cite[Corollary~8.1]{KM-unknot}.

\begin{corollary}\label{cor:spectral-sequence-R3}
For webs $K_{v}$ as above, there is a spectral sequence of
$\cS$-modules whose $E_{1}$ term is
\[
          \bigoplus_{v \in \{0,1\}^{N}} 
                \Isharp(K_{v} ;\Gamma_{\s})
\]
and which abuts to the instanton Floer homology $\Isharp(K_{v} ; \Gamma_{\s})$,
for $v=(2,\dots,2)$.  The differential $d_{1}$ is the sum of the maps
induced by the cobordisms $\Tigma_{vu}$ with $v>u$ and $|v-u|=1$. \qed
\end{corollary}

\subsection{Pants and the \texorpdfstring{$E_{2}$}{E2} page}

The spectral sequence in Corollary~\ref{cor:spectral-sequence-R3} is
set up quite generally
for webs in a fixed 3-manifold, differing by skein moves inside fixed
balls. The standard application for this setup is to consider a plane
projection and have the fixed balls correspond to the crossings in the
projection.

So let $K$ be a link in $\R^{3}\subset S^{3}$ with a planar projection
giving a diagram $D$ in $\R^{2}$. Let $N$ be the number of crossings
in the diagram. As in \cite{Khovanov}, we can consider the $2^{N}$
possible smoothings of $D$, indexed by the points $v$ of the cube
$\{0,1\}^{N}$. The conventions we use for the labels $\{0,1\}$ is the
same as the convention in
\cite{Khovanov,Rasmussen-slice}, and is also consistent with the convention illustrated in
Figure~\ref{fig:Tetrahedra-skein}. The smoothings give  $2^{N}$ different
unlinks $K_{v}$ in the plane of the projection. For each $v\ge u$ in $\{0,1\}^{N}$,
we have our standard cobordism $\Tigma_{vu}$ from $K_{v}$ to $K_{u}$,
with its family of metrics.

We apply Corollary~\ref{cor:spectral-sequence-R3} to this
situation. We learn that there is a spectral sequence abutting to
$\Isharp(K ; \Gamma_{\s})$ whose $E_{1}$ term is
\[
   E_{1}= \bigoplus_{v\in \{0,1\}^{N}} \Isharp(K_{v} ; \Gamma_{\s}).
\]
and whose differential $d_{1}$ is
\begin{equation}
\label{eq:d1-sum}
     d_{1} = \sum_{|v-u|=1} 
     \Isharp(\Tigma_{vu} ; \Gamma_{\s}).
 \end{equation}

In this situation,  unlike the general case considered previously,
 each cobordism $\Tigma_{vu}$ with $|v-u|=1$ is a cobordism between
 \emph{planar unlinks}, obtained from a ``pair of
pants'' that either joins two components into one, or splits one
component into two. We have already computed
$\Isharp(U_{n};\Gamma_{\s})$ for a planar unlink $U_{n}$
(Corollary~\ref{cor:unlink-canonical}) as well as the maps that the
pants and copants cobordisms (section~\ref{subsec:pants-copants}). So
we completely understand the $E_{1}$ page and its differential
$d_{1}$. We have
\[
   E_{1}= \bigoplus_{v\in \{0,1\}^{N}} V^{\otimes n(v)}
\]
where $V$ is a free $\cS$-module of rank $2$, admitting a standard
basis $\bx_{+}$, $\bx_{-}$, and $n(v)$ indexes the components of
the unlink $K_{v}$. Whenever $v> u$ and $|v-u|=1$, the corresponding
summand of $d_{1}$ in \eqref{eq:d1-sum} involves only the factors $V$
of the tensor product that are adjacent to the vertex at which $v$ and
$u$ differ, where it is given by
\begin{equation}\label{eq:Frob-cS-1}
             \Isharp(\pants; \Gamma_{\s}) : V\otimes V\to V
\end{equation}
or
\begin{equation}\label{eq:Frob-cS-2}
            \Isharp(\copants; \Gamma_{\s}) :V \to V\otimes V
\end{equation}
depending on whether two components of $K_{v}$ merge in $K_{u}$, or
one component splits.

In the language of \cite{Khovanov-Frobenius}, the $\cS$-module $V$
equipped with the multiplication $\Isharp(\pants; \Gamma_{\s})$ and
comultiplication $\Isharp(\copants; \Gamma_{\s})$ is a self-dual, rank-2 Frobenius
system $\cF_{\s}$. As an algebra, its unit element is $\bx_{+}$, and its
co-unit is $\by_{+}$ (Definition~\ref{def:xy-U1}). The
multiplication is described completely by giving the square of the
element $x=\bx_{-}$, the formula for which is in \eqref{eq:multn}. So we can write it as
\[
        \cS[x] \big/ \bigl( x^{2} + \s(P) x + \s(Q) \bigr).
\]
Our description of $(E_{1}, d_{1})$ above coincides
with Khovanov's definition of the complex that computes the knot
homology group corresponding to this Frobenius system. There is only
the slight change of conventions, because of the historically reversed
roles of the two smoothings $\{\,0,1\,\}$. With that understood, we
can identify the $E_{2}$ page of the spectral sequence:

\begin{proposition}
    In the special case that the cube of resolutions is the one
    obtained from a planar diagram of a knot or link $K$,
    the $E_{2}$ page of the spectral sequence in
    Corolllary~\ref{cor:spectral-sequence-R3} is isomorphic to the
    knot homology $H(\bar{K}; \cF_{\s})$ in the notation of
    \cite{Khovanov-Frobenius}, where $\cF_{\s}$ is the rank-2 Frobenius
    system over $\cS$ given by the multiplication \eqref{eq:Frob-cS-1}
    and comultiplication \eqref{eq:Frob-cS-2}. Here $\bar K$ denotes
    the mirror image of $K$. \qed
\end{proposition}

\begin{corollary}\label{cor:ss-skeleton}
    For a knot or link $K$ in $\R^{3}$, there is a spectral sequence
    whose $E_{2}$ page is Khovanov's homology  $H(\bar{K}; \cF_{\s})$
    corresponding to the Frobenius system $\cF_{\s}$ and which abuts to
    the instanton homology with local coefficients, $\Isharp(K;
    \Gamma_{\s})$. \qed
\end{corollary}

Theorem~\ref{thm:F5-ss} in the introduction, along with its two
corollaries, are obtained directly from
Corollary~\ref{cor:ss-skeleton} by identifying the Frobenius system
$\cF_{\s}$ in each case, to compare it with
those described in the notation of \cite{Khovanov-Frobenius}. We begin
with the
case that $\cS=\cR$ and $\s=1$ (i.e.~the case of the local system
$\Gamma$). Here the resulting Frobenius system $\cF$ corresponds to
\[
        \cR[x] \big/ \bigl( x^{2} + P x + Q \bigr).
\]
As explained in the introduction, the universal example from
\cite{Khovanov-Frobenius}, when reduced mod $2$, is a Frobenius
system $F_{5}$ over
\[
             R_{5} = \F_{2}[h,t].
\]
Its multiplication is given by
\[
            R_{5}[x] \big/ \big( x^{2} + h x + t \bigr).
\]
Since the comultiplications can be compared similarly, we see that the
Frobenius $\cF$ arising from $\Isharp$ with coefficient system
$\Gamma$ is $F_{5}\otimes_{r} \cR$, where $r$ maps $h$ to $P$ and $t$
to $Q$. Theorem~\ref{thm:F5-ss} is therefore a consequence of
Corollary~\ref{cor:ss-skeleton}.

Corollaries~\ref{cor:BN-ss} and \ref{cor:barBN-ss} follow from this
universal version by base change, as explained in the introduction.

\subsection{The spectral sequence for reduced homologies}

There is also a version of Corollaries~\ref{cor:BN-ss} and
\ref{cor:barBN-ss} for the reduced homology theories. Given a link
with a base-point, and given a diagram for the link such that the
base-point does not lie at a crossing, we may form again the cube of
resolutions, and for each vertex of the cube we now have a planar
unlink with a single marked point. Let $\s:\cR\to\cS$ be a base change
with $\s(T_{0}) = \s(T_{1})$, so that the reduced theory $\Inat$ is
defined. The basic spectral sequence described in
Corollary~\ref{cor:spectral-sequence-R3} has a reduced counterpart, whose statement and
proof are essentially the same:

\begin{proposition}\label{prop:reduced-spectral-sequence-R3}
There is a spectral sequence of
$\cS$-modules whose $E_{1}$ term is
\[
          \bigoplus_{v \in \{0,1\}^{N}} 
                \Inat(K_{v} ;\Gamma_{\s})
\]
and which abuts to the instanton Floer homology $\Isharp(K_{v} ; \Gamma_{\s})$,
for $v=(2,\dots,2)$.  The differential $d_{1}$ is the sum of the maps
induced by the cobordisms $\Tigma_{vu}$ with $v>u$ and $|v-u|=1$. \qed
\end{proposition}

The condition that $\s(T_{0})=\s(T_{1})$
implies that the Frobenius system $\cF_{\s}$ has a description in
which the algebra is
\[
             \cS[M] / (M^{2} + \s(P) M)
\]
and the comultiplication is given by
\[
      \begin{aligned}
        1&\mapsto 1\otimes M + M\otimes 1 + \s(P)(1\otimes 1) \\
       M& \mapsto M\otimes M.
    \end{aligned}
\]

In such a situation, there is a reduced link homology $\tilde H (\bar
K ; \cF_{s})$ obtained from
the cube of resolutions. It is defined from a complex $\tilde{C}$ for
which the contribution $\tilde C_{v}$ from a vertex of the cube is
\[
          \tilde \cA \otimes_{\cS} \cA \otimes_{\cS} \cdots \otimes_{\cS} \cA
\]
where $\cA$ is the Frobenius algebra of $\cF_{\s}$ and $\tilde \cA$ is
the $\cS$-submodule generated by $m$. The tensor product is over all components of the unlink $K_{v}$
and the factor $\tilde \cA$ corresponds to the component with the
basepoint. The edge maps as usual come from the multiplication and
comultiplication, restricted to $\tilde \cA \subset \cA$ if
necessary.

Using Proposition~\ref{prop:reduced-calc} and
Corollary~\ref{cor:reduced-calc} for the edges involving the component
with the base point, we can match up the differential $d_{1}$ in
Proposition~\ref{prop:reduced-spectral-sequence-R3} with the
multiplication and comultiplication maps of $\tilde \cA\otimes \cA \to
\tilde \cA$ and $\tilde \cA \to \tilde \cA \otimes \cA$.

We obtain in this way a reduced counterpart to Corollary~\ref{cor:BN-ss}.

\begin{corollary}\label{cor:BN-ss-reduced}
There is spectral sequence of modules over the Laurent series ring
$\cS_{\BN}$ in three variables,  from the reduced version of graded Bar-Natan homology in
characteristic $2$,
\[
     \widetilde{\BN}(\bar K) \otimes_{r_{1}} \cS_{\BN} \implies \Inat(K ; \Gamma_{\BN}),
\]
to the reduced instanton homology group with coefficients in the local system 
$\Gamma_{\BN}= \Gamma \otimes_{\s_{\bn}}\cS_{\BN}$, where the base
change $\s_{bn}$ is given by \eqref{eq:sbn-base-change}. \qed
\end{corollary}

The reduced version of Corollary~\ref{cor:barBN-ss}, for filtered
Bar-Natan homology, can be formulated in the same way:

\begin{corollary}\label{cor:barBN-ss-reduced}
For a knot or link $K$, let $\widetilde{\FBN}(K)$ denote the
reduced version of filtered
Bar-Natan homology over $\F_{2}$. Then there is a spectral sequence of
vector spaces over\/ $\F_{4}$,
\[
  \widetilde{\FBN}(\bar K)\otimes \F_{4} \implies \Inat(K ; \Gamma_{\FBN}),
\]
where $\Gamma_{\FBN}$ is the local system of\/ $\F_{4}$ vector
obtained from $\Gamma$ by the base change \eqref{eq:filtered-BN-s}.
\qed
\end{corollary}

\bibliographystyle{abbrv}
\bibliography{ibn}

\end{document}